\documentclass[final,3p,times]{elsarticle}

\usepackage{amssymb}
\usepackage{amsmath}
\usepackage{color}
\usepackage{amsthm}
\usepackage{subfig}
\usepackage{hyperref}
\usepackage{algorithm}%
\usepackage{algorithmicx}%
\usepackage{algpseudocode}%

\newtheorem{theorem}{Theorem}
\newtheorem{lemma}[theorem]{Lemma}
\newtheorem{remark}{Remark}

\newcommand{\bfx}{\boldsymbol{x}}

\newcommand{\mO}{\mathcal O}
\newcommand{\mI}{\mathcal{I}_h}
\newcommand{\mT}{\mathcal{T}^h}

\journal{AABB}

\begin{document}

\begin{frontmatter}

\title{HoSGFEM: High-order stable generalized finite element method for elliptic interface problem}

\author[1]{Bingying Zhao}
\author[1]{Yin Song}
\author[2,3]{Quanling Deng}
\author[1]{Xin Li\corref{note}}

\cortext[note]{Corresponding author}

\ead{lixustc@ustc.edu.cn}

\address[1]{School of Mathematical Sciences, University of Science and Technology of China, Hefei, China}
\address[2]{Yau Mathematical Sciences Center, Tsinghua University, Beijing 100084, China}
\address[3]{School of Computing, Australian National University, Canberra, ACT 2601, Australia}

\begin{abstract}
The Generalized Finite Element Method (GFEM) is an effective unfitted numerical method for handling interface problems. By augmenting the standard FEM space with an appropriate enrichment space, GFEM can accurately capture $C^0$ solutions across the interfaces. While
numerous GFEMs for interface problems have been studied, establishing a stable high-order GFEM with optimal convergence rates and robust system conditioning remains a challenge. 
The highest known order of two was established by Zhang and Babuška (SGFEM2, Comput. Methods Appl. Mech. Engrg. 363 (2020), 112889).
In this paper, we propose a unified enrichment space construction and establish arbitrary high-order stable GFEMs (HoSGFEM) for elliptic interface problems. 
The main idea distinguishes itself from Zhang and Babuška's SGFEM2 substantially
and it is twofold:  a) we construct dimensionality-reduced auxiliary locally supported piecewise polynomials that satisfy the partition of unity property for elements containing interfaces; 
b) we construct the enrichment scheme based on $d\{1,(x-x_c^e), \cdots, (y-y_c^e)^{p-1} \}$ ($d$ is the distance function; $(x_c^e, y_c^e)$ is the center of the element containing interface, thus element-based) for arbitrary $p$-th order elements instead of $d, d\{1,x,y\}$ or $d\{1,x,y,x^2,xy,y^2\}$ (global functions) for $p=1,2$ in the literature. 
This idea results in an enrichment space that has a large angle with the standard FEM space, leading to the stability of the method with system condition number growing in order $\mO(h^{-2})$.
We establish optimal convergence rates for HoSGFEM solutions under the proposed construction. Various numerical experiments with both straight and curved interfaces demonstrate the optimal convergence, FEM-comparable system condition number with $\mO(h^{-2})$ growth,  and robustness as element boundaries approach interfaces.  
\end{abstract}


\end{frontmatter}

\section{Introduction}\label{sec:intr}
Interface problems are ubiquitous in science and engineering, and solving them poses a number of challenges, including those in composite material elasticity~\cite{reddy2003mechanics, boresi2010elasticity}, cell problem for complex microstructures~\cite{moes2003computational, legrain2012high} and fluid dynamics of two-phase flows~\cite{folch1999phase, sauerland2011extended}. Because of different material properties or governing parameters, i.e., the discrepancy in physical coefficients across the governing equations on different subdomains, the solution typically loses smoothness at the interface. The standard finite element method (FEM)  requires sufficiently refined meshes to achieve the target accuracy or a fitted mesh such that its element boundaries are aligned with the interface when addressing such problems. However, when dealing with complex geometries and evolving interfaces, generating a fitted mesh becomes challenging and time-consuming. Therefore, unfitted mesh methods have garnered substantial research interest over the last several decades, as they use fixed meshes that are independent of interface position, thereby significantly reducing the complexity of mesh generation.

The unfitted mesh method mainly includes two approaches: the penalty method and the enrichment space method. The penalty method in general splits each basis function across the interface into two basis functions and adds to the standard bilinear form a penalization bilinear term acting on these splitted basis functions~\cite{hansbo2002unfitted}. The primary challenge in penalty method implementation is the small cut cell problem~\cite{strouboulis2000design, huang2017unfitted, chen2021adaptive, chen2023arbitrarily, chen2024arbitrarily, johansson2013high}. The other unfitted approach uses enrichment space to approximate the solution with reduced continuity at the interfaces. This includes series of generalized finite element methods (GFEMs)~\cite{kergrene2016stable,babuvska2017strongly, zhang2019strongly, zhang2020stable,deng2020higher},
cutFEMs~\cite{burman2015cutfem,burman2022cutfem,burman2025cut,burman2018shape}, immersed FEM~\cite{li1998immersed,zhang2004immersed,liu2007mathematical}, 
 and generalized isogeometric analysis (GIGA) ~\cite{jia2015extended, tan2015nurbs, zhang2022generalized, hu2024higher,song2025stable}. 
%

We focus on GFEMs that holds the partition of unity (PU) property~\cite{melenk1996partition, babuvska1997partition, duarte1996hp}. As an extension of the standard FEM, GFEM enhances approximation accuracy by incorporating local enrichment functions within the finite element space to capture discontinuities or singularities. Therefore, the GFEM can effectively handle complex problems such as material interfaces and crack propagation without the need of generating fitted meshes. It is worth noting that the extended finite element method (XFEM)~\cite{moes1999finite, daux2000arbitrary} was developed independently from GFEM but shares significant similarities with it. Both methods are based on the PU framework and incorporate local enrichment functions within the solution space. We use GFEM to refer to both methods in this paper.

Although numerous theoretical and numerical studies confirm that GFEM delivers higher approximation accuracy than standard FEM for problems with non-smooth solutions, the theoretical and practical development of a GFEM with numerical stability and robustness remains challenging. 
For instance, a series of work~\cite{zhang2020stable, nicaise2011optimal, babuvska2012stable, zhang2022condensed} have investigated the solution convergence property of GFEMs. However, many of these enrichment schemes result in a much higher condition number for the stiffness matrix. This increase in condition number amplifies the impact of rounding errors when solving the linear system, leading to possible loss of numerical stability. Ensuring the numerical stability of GFEM is crucial. A desirable property of a stable formulation is that the condition number of its stiffness matrix grows at a comparable rate as in the standard FEM, i.e., $\mO(h^{-2})$.
In summary,
the desired properties of GFEMs are threefold:  
\begin{itemize}
\item (a) \textit{convergence}--it achieves optimal convergence rates; 
\item (b) \textit{numerical stability}--its stiffness-matrix condition number grows at the FEM-comparable rate of $\mO(h^{-2})$; and
\item (c) \textit{robustness}--this conditioning does not worsen as element boundaries approach interfaces.  
\end{itemize}

Achieving all three properties requires careful construction of enrichment functions to prevent ill-conditioning and ensure numerical stability.
A GFEM that meets these criteria is termed a stable GFEM (SGFEM).
 To this end, various strategies have been proposed for numerical stability, such as matrix preconditioning~\cite{schweitzer2011stable, menk2011robust}, orthogonalization of enrichment functions~\cite{agathos2019improving, agathos2019unified}, modification of PU functions~\cite{oh2008piecewise, hong2013mesh} and incorporation of interpolation operators within the enrichment construction~\cite{zhang2020stable, babuvska2012stable}.
To the best of our knowledge, linear SGFEM for multidimensional interface problems has been extensively developed, and its rigorous theoretical foundations, including proofs of convergence and stability, are now well established in the literature; see, for example, \cite{kergrene2016stable,babuvska2017strongly,cui2022stable}.  

In 1D, the linear SGFEM framework of \cite{kergrene2016stable,babuvska2017strongly} extends naturally to higher-order elements; see \cite{deng2020higher}. In 2D, the highest known convergence order for SGFEM, namely order two, was established by Zhang and Babuška in 2020 \cite{zhang2020stable}. The work developed a second-order GFEM that achieves optimal convergence for problems with numerical stability and robustness by employing cubic Hermite polynomials together with local principal component analysis (LPCA) to control the stiffness-matrix condition number. This second-order construction of the enrichment space differs substantially from the linear SGFEM formulation. In three or more dimensions, no high-order ($p\ge 2$) SGFEM for interface problems has been reported to date.

Generalizing SGFEM to higher orders ($p>2$) is considerably more intricate. Such methods continue to encounter the previously noted difficulties of ensuring the three conditions: optimal convergence, numerical stability, and robustness. Research on higher-order GFEM remains relatively limited, and fully satisfactory theoretical and numerical results have yet to be achieved. For example, several higher-order formulations were proposed in \cite{cheng2010higher} (in the XFEM setting), but their numerical experiments reported only near-optimal convergence rates.

In this paper, we focus on 2D interface problem and propose a uniform construction of the enrichment spaces for arbitrary order $(p \ge 1)$ including the linear case (hence, a new linear SGFEM as a byproduct), and we refer to our method as high-order stable GFEM (HoSGFEM). 
The main idea is twofold. 
Firstly, for the elements containing interfaces, we adopt the Bernstein polynomials with local supports and form a linear combination, denoted as $\phi$, of them as basis functions, where the combination coefficients are determined in a way that reduces the linear dependence of the functions in the enrichment space. 
This construction satisfies the PU property.
This combination reduces the dimensionality of the enrichment space and it is the key to controlling the conditioning of the resulting stiffness system. 
Secondly, instead of using global enrichment functions $d, d\{1,x,y\}$ or $d\{1,x,y,x^2,xy,y^2\}$ for $p=1,2$ in the literature, we construct element-centroid-based local enrichment functions using $d\{1,(x-x_c^e), \cdots, (y-y_c^e)^{p-1} \}$ for $p$-th order elements. 
Herein, $d$ is the distance function, i.e., the absolute value of the distance of a point $(x,y)$ in the domain to the closed point on the interface. $(x_c^e, y_c^e)$ is the centroid of the element containing part of the interface. 
Finally, the local enrichment basis function is of the form $\phi\big( d(x-x_c^e)^i(y-y_c^e)^j - \mI(d(x-x_c^e)^i(y-y_c^e)^j) \big), \ 0\le i, j <p$, where $\mI$ is the interpolation operator for the FEM space.
This construction of the element-centroid-based enrichment functions reduces the linear dependence of the basis functions in the enrichment space with the basis functions in the FEM space. 
As a result, this idea constructs an enrichment space that has a large angle with the standard FEM space, leading to the numerical stability with system condition number growth controlled in order $\mO(h^{-2})$.
We show that the method satisfies the three conditions (a)-(c), hence, indeed a SGFEM. 
Specifically, we establish optimal convergence rates for the proposed HoSGFEM. 
A set of numerical experiments, including test problems with both straight and curved interfaces, confirms optimal convergence, numerical stability with stiffness-matrix condition number that scales comparably to the standard FEM, and robustness as element boundaries approach the interfaces. 
This paper focuses on rectangular elements in two dimensions. Extension to triangular elements and to three-dimensional interface problems will be explored in future work.

The rest of this paper is organized as follows. Section \ref{section2} introduces the model problem. The GFEM and several constructions of enriched space are discussed in Section \ref{sec:sgfem}. In section \ref{sec:hosgfem}, we provide the detailed construction of a new enriched space, including PU functions and enrichment functions. Section \ref{sec:ea} provides the proof of the optimal convergence property. Section \ref{sec:num} presents the results of several numerical experiments, and the last section gives a summary of this work.

\section{Elliptic interface problems}\label{section2}
We focus on the two-dimensional elliptic interface problem. 
Let $\Omega \subset \mathbb{R}^2$ be a bounded, simply connected, and convex domain with Lipschitz boundary $\partial \Omega$. 
$\Gamma$ is a smooth interface which divides the $\Omega$ into two subdomains $\Omega_0$ and $\Omega_1$ such that $\bar{\Omega}=\bar{\Omega}_0 \cup \bar{\Omega}_1$, $\Omega_0 \cap \Omega_1=\emptyset$ and $\Gamma = \bar{\Omega}_0 \cap \bar{\Omega}_1$. Figure \ref{fig:Interface problem} shows a square domain with an open interface. The second-order elliptic problem with Neumann boundary conditions on $\Omega$ is: Find $u$ such that
\begin{equation}\label{eq:pde}
\begin{aligned}
    -\nabla  \cdot (\kappa \nabla u) & = f, \quad \text{in}\ \Omega,\\
    \frac{\partial (\kappa \nabla u)}{\partial \boldsymbol{n}_{b}} & = g, \quad \text{on} \ \partial\Omega,\\
    \llbracket u \rrbracket_{\Gamma} & = 0, \\
    \llbracket \kappa  \frac{\partial u}{\partial \boldsymbol{n}} \rrbracket_{\Gamma} & = q, \quad \text{on}\ \Gamma,
\end{aligned}
\end{equation}
where $f,g,q$ are given functions satisfying the compatibility condition:
$\int_\Omega f d \bfx + \int_{\partial \Omega} g ds + \int_{\Gamma} q ds = 0.$
Herein, $\kappa$ is a positive, piecewise function which is $\kappa_0$ in $\Omega_0$ and $\kappa_1$ in $\Omega_1$ with $\kappa_0 \ne \kappa_1$. 
$\boldsymbol{n}_{b}$ is the unit outward normal to the boundary $\partial \Omega$. 
Let $u_j, j=0,1,$ denote the restrictions of $u$ on $\Omega_j$.
$ \llbracket u \rrbracket_{\Gamma}$ denotes the jumps of function $u$ across the interface $\Gamma$. 
Specifically, the jumps $\llbracket u \rrbracket_{\Gamma}$ and $\llbracket \kappa \frac{\partial u}{\partial \boldsymbol{n}} \rrbracket_{\Gamma}$ are defined as
\begin{equation*}
    \begin{aligned}
        \llbracket u \rrbracket_{\Gamma} &:= u_0(\boldsymbol{x}) - u_1(\boldsymbol{x}),&& \boldsymbol{x} \in \Gamma\\
        \llbracket \kappa \frac{\partial u}{\partial \boldsymbol{n}} \rrbracket_{\Gamma} &:= \kappa_0 \frac{\partial u_0}{\partial \boldsymbol{n}_0} (\boldsymbol{x}) + \kappa_1 \frac{\partial u_1}{\partial \boldsymbol{n}_1}(\boldsymbol{x}), && \boldsymbol{x} \in \Gamma,
    \end{aligned}
\end{equation*}
where $\boldsymbol{n}_0$ and $\boldsymbol{n}_1$ are the unit outward normal to the interface $\Gamma$ as the intersection at the boundary of $\Omega_0$ and $\Omega_1$.
With the Neumann boundary conditions on $\Omega$, the solution to the interface problem \eqref{eq:pde} is unique up to a constant. 
To ensure uniqueness, we impose the normalization condition that the integral of the solution over the domain vanishes.

\begin{figure}[htbp]
    \centering
    \includegraphics[width=0.3\linewidth]{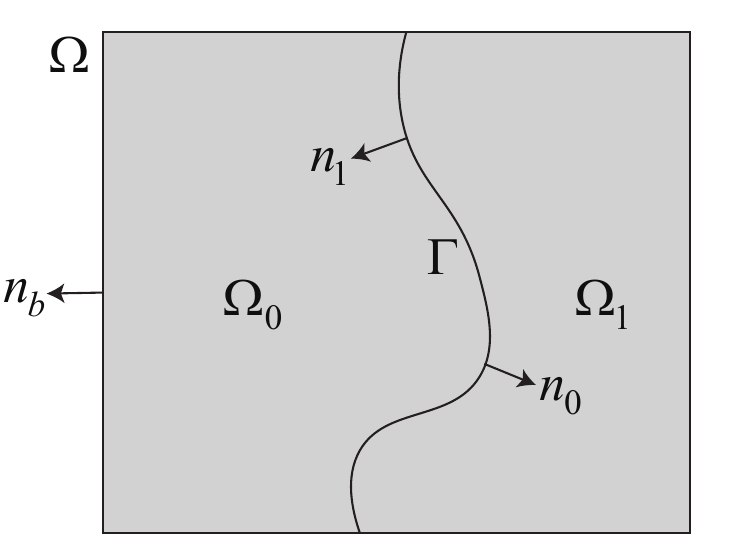}
    \caption{An example of a domain $\Omega$ with a curved interface $\Gamma$.}
    \label{fig:Interface problem}
\end{figure}

Let $V := \{ w\in H^1(\Omega):  \llbracket w \rrbracket_{\Gamma} = 0\}$ be the tailored Hilbert space for the interface problem~\eqref{eq:pde}. Define the bilinear form $a(\cdot, \cdot)$ on $V \times V$ and linear functional $l(\cdot)$ on $V$ as follows
\begin{equation}\label{eq:BiLinearForm}
    a(u, v) = \int_{\Omega} \kappa \nabla u \cdot \nabla v d \boldsymbol{x},\quad l(v) = \int_{\Omega} fv d \boldsymbol{x} + \int_{\partial \Omega} g v d s + \int_{\Gamma} q v ds.
\end{equation}
The variational problem of the equation~\eqref{eq:pde} is: Find 
$u \in V$ such that 
$$a(u, v) = l(v),\ \forall v \in V.$$
%
%
Let $\mT$ be a partition of $\Omega$ into non-overlapping mesh cells, namely quadrilaterals. 
Let $V^h \subset V$ be a finite-dimensional test and trial space.
The Galerkin formulation of \eqref{eq:pde} is to
find $u_h \in V^h$ such that
\begin{equation}\label{eq:GalApproProm}
    a(u_h, v) = l(v),\ \forall v \in V^h,
\end{equation}
where the discrete space $V^h$ will be specified below for each method within the Galerkin framework.

\section{HoSGFEM main idea}\label{sec:sgfem}

In this section, we first present the general framework of the GFEM, then review the related and widely-studied GFEMs, which are stable in lower-order cases, and finally introduce the proposed high-order stable formulation, referred to as HoSGFEM.

\subsection{GFEM framework}
Consider a quasi-uniform quadrilateral partition $\mathcal{T}^h$ of domain $\Omega$ with  elements $e_i$ such that $\mathcal{T}^h = \{e_i\}$.
In this paper, we refer to elements intersected by the interface as \emph{cut elements} or \emph{interface elements}. 
GFEM enriches the standard FEM approximation space by incorporating a local enrichment space. For an interface problem, the enrichment space typically has local support over the elements containing and surrounding the interfaces. Denote the standard FEM space
as $V^h_{FEM}$ and the enrichment space as $V^h_{ENR}$.
The enrichment space $V^h_{ENR}$ is spanned by a set of additional local functions that captures the feature of the solution at the interface $\Gamma$.
GFEM approximation space is defined as $V^h_{GFEM}$, where
\begin{equation}\label{eq:vhgfem}
    V^h_{GFEM} := V^h_{FEM} \oplus V^h_{ENR} = \left\{u = u_1 + u_2: u_1 \in V^h_{FEM}, u_2 \in V^h_{ENR} \right\}.
\end{equation}
The enrichment space admits a representation of the form
$
    V^h_{ENR} = \text{span}\{\psi_i, i \in \mathit{I}_{ENR}\},
$
where $\mathit{I}_{ENR}$ is the index set of enrichment nodes. A basis function $\psi_i$ in $V^h_{ENR}$ is typically written as a product of a function in $V^h_{FEM}$ and an enrichment function that captures the solution feature at the interface.
The number of the enrichment functions $n_{p}$ depends on the order $p$ and the mesh configuration. Following the aforementioned construction, the numerical approximation $u_h$ in the GFEM space takes the form
\begin{equation*}
    u_h = \sum_{i \in I_{FEM}}  \phi^p_i u^h_{1,i} + \sum_{i \in I_{ENR}}  \psi_{i} u^h_{2,i},
\end{equation*}
where 
$
    V^h_{FEM} = \text{span}\{\phi_i^p,i\in \mathit{I}_{FEM}\} 
$
and $\mathit{I}_{FEM}$ is the index set of the $p$-th order standard FEM space.
Similar to standard FEM, the corresponding coefficient vector $\mathbf{U} = [\mathbf{U}_1, \mathbf{U}_2]$ can be obtained by
solving the linear system
\begin{equation*}
	\mathbf{KU} :=
    \left[\begin{array}{ll}
    \mathbf{K}_{11} & \mathbf{K}_{12} \\
    \mathbf{K}_{12}^T & \mathbf{K}_{22}
    \end{array}\right]
    \left[\begin{array}{l}
    \mathbf{U}_1 \\
    \mathbf{U}_2
    \end{array}\right]=
    \left[\begin{array}{l}
    \mathbf{F}_1 \\
    \mathbf{F}_2
    \end{array}\right] =: \mathbf{F},
    \end{equation*}
where
\begin{equation*}
\begin{aligned}
& \left(\mathbf{K}_{11}\right)_{i,j}=a\left( \phi^p_j, \phi^p_i\right),\left(\mathbf{K}_{2 2}\right)_{i,j}=a\left( \psi_j,  \psi_i\right), \left(\mathbf{K}_{12}\right)_{i, j}= a\left(\psi_j,  \phi^p_i\right),\\
& \left(\mathbf{F}_1\right)_{i} = l\left( \phi^p_i \right),  \left(\mathbf{F}_2\right)_{j} = l\left(\psi_j\right).
\end{aligned}
\end{equation*}
Here bilinear form $a(\cdot, \cdot)$ and linear functional $l(\cdot)$ is defined in~\eqref{eq:BiLinearForm}. In practical implementation, the above linear system sometimes is preconditioned, leading to the following linear system
\begin{equation*}
 \widehat{\mathbf{K}} \widehat{\mathbf{U}}=\widehat{\mathbf{F}},\quad \widehat{\mathbf{K}}=\mathbf{D K D}, \quad \widehat{\mathbf{F}}=\mathbf{D F}, \quad \widehat{\mathbf{U}}=\mathbf{D}^{-1} \mathbf{U},
\end{equation*}
where $\mathbf{D}_{ii} = (\mathbf{K}_{ii})^{-1/2}$. The condition number of $\widehat{\mathbf{K}}$ is referred to scaled condition number (SCN) in the literature.

\subsection{Overview of GFEMs} \label{sec:m}
In this section, we provide an overview of the widely studied GFEMs for interface problems. These methods are well established and exhibit the desired properties (a), (b), and (c) outlined in Section \ref{sec:intr} for lower-order elements. Although they can be extended to higher-order elements in a straightforward manner, certain desired properties are no longer preserved. This limitation will be illustrated through numerical experiments, serving as a comparison with the proposed method.

We first introduce some notations. Let $I_{el}$ denote the index set of all elements and $I^0_{el}:= \{i: e_i \cap \Gamma \neq \emptyset, i \in I_{el} \}$ denote the interface elements. 
The index set of the union of interface elements and their adjacent elements is denoted as $I^1_{el}:=\{i: e_i \cap e_j \neq \emptyset, j \in I^0_{el} \}$. Moreover, we denote by $I_v$ and $I_n$ the index sets of mesh vertices $v_i$ and finite element nodes $\boldsymbol{x}_i$, respectively. Figure \ref{fig:CartesianMesh} uses different markers to illustrate this setting for the case of quadratic elements.
\begin{figure}[ht]
    \centering
    \includegraphics[width=0.35\linewidth]{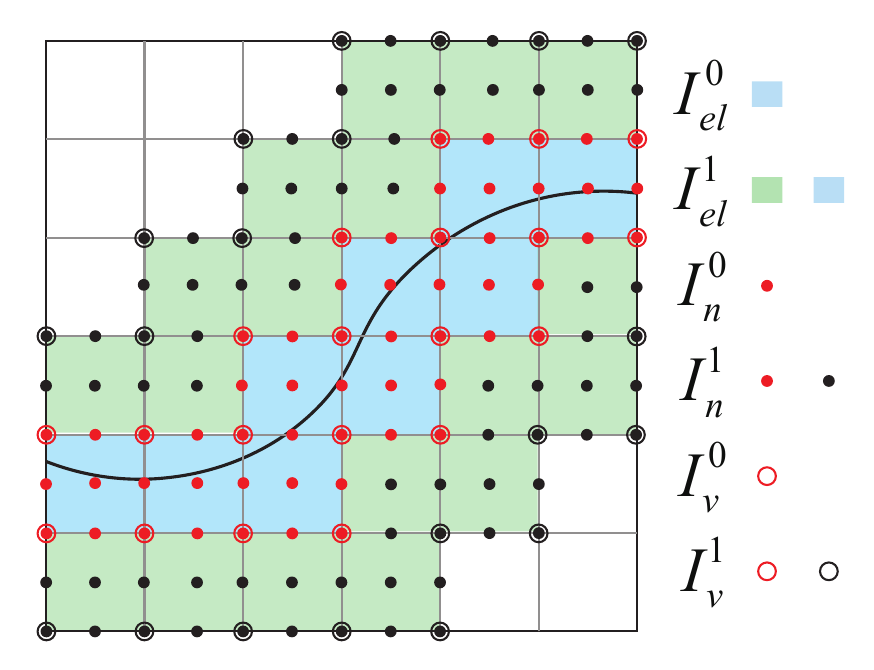}
    \caption{Illustration of different index sets.}
    \label{fig:CartesianMesh}
\end{figure}

Now, we review the existing enrichment space construction for GFEMs.

\textbf{GFEM:} A common construction of enrich space for interface problems is given by
\begin{equation}\label{eq:GFEM}
    V^h_{ENR} = \text{span} \left\{\phi^1_i(\boldsymbol{x}) d(\boldsymbol{x}), i \in \mathit{I}^0_{n} \right\},
\end{equation}
where $\phi^1_i(\boldsymbol{x})$ is the standard linear FE shape functions and $d(\boldsymbol{x}) = d_{\Gamma}(\boldsymbol{x}):=d(\boldsymbol{x}, \Gamma)$ is the
distance function associated with interface $\Gamma$. 
The set of enriched nodes $\mathit{I}^0_{n}$ are composed of nodes of all the interface elements. Figure \ref{fig:CartesianMesh} marks the enriched
nodes in set $\mathit{I}^0_{n}$ with red dots. 
The construction is simple and can easily extended to higher-order elements by replacing $\phi^1_i$ by $\phi^p_i$. 
However, the above construction for linear elements can only achieve the suboptimal convergence rate; we refer to \cite{kergrene2016stable}.

\textbf{Corrected GFEM:} Recent studies\cite{chessa2003construction, fries2008corrected} found that the suboptimal convergence of
construction~\eqref{eq:GFEM} is due to the lack of PU property on elements adjacent to interface elements. These elements are called
\emph{blending elements} where a portion of their nodes are used as enriched nodes. Corrected GFEM \cite{cheng2010higher, fries2008corrected}
addresses this issue by introducing a ramp function
\begin{equation*}
    R(\boldsymbol{x}) = \sum_{i \in I^0_{n}} \phi^1_i(\boldsymbol{x})
\end{equation*}
into the enrich space
\begin{equation}\label{eq:CorrGFEM}
    V^h_{ENR} =\text{span} \left\{\phi^1_i(\boldsymbol{x}) d(\boldsymbol{x}) R(\boldsymbol{x}), i \in \mathit{I}^1_{n} \right\},
\end{equation}
where $\mathit{I}^1_{n}$ are the nodes on interface elements and blending elements. Figure \ref{fig:CartesianMesh} illustrates all enriched nodes
in set $\mathit{I}^1_{n}$, marked jointly by red and black dots. The increased number of enriched nodes ensures complete partition of unity
functions on blending elements. Meanwhile, since the support of ramp functions $R(\boldsymbol{x})$ is confined to interface elements and
blending elements only, regions beyond these elements remain unaffected. However, even with these improvements, optimal convergence rates
still cannot be guaranteed.

\textbf{Stable GFEM:} Another issue for GFEM is the faster-growing scaled condition numbers than standard FEM. Standard FEM shows a condition growth of rate
$\mO(h^{-2})$ when performing $h$-refinement, while GFEM often yields considerably higher rate. The occurrence of this phenomenon can
be partially attributed to the excessively high linear correlation between the added enrichment basis functions and the finite element basis
functions, which makes the angle between the enriched space $V^h_{ENR}$ and finite element space $V^h_{FEM}$ too small. Stable GFEM try to address
this issue by incorporating the interpolation of the distance function, leading to the enriched space as 
\begin{equation}\label{eq:SGFEM}
    V^h_{ENR} =\text{span} \left\{\phi^1_i(\boldsymbol{x}) (d(\boldsymbol{x}) - \mathcal{I}_h (d)(\boldsymbol{x})), i \in \mathit{I}^{0}_{n} \right\},
\end{equation}
where $\mathcal{I}_h (w) = \sum_{i \in I_{n}} w(\boldsymbol{x}_i) \phi^1_i(\boldsymbol{x})$ is the interpolation operator. This technique has been adopted in several existing studies~\cite{zhang2020stable, babuvska2012stable, kergrene2016stable, sauerland2013stable, cui2022stable}. The enrich space achieves local approximation of $u-\mathcal{I}^h u$ and can be theoretically proven to obtain optimal convergence for linear elements. Additionally, the method is robust and condition numbers do not grow faster than standard FEM. This method satisfies the desired properties (a), (b), and (c) outlined in Section \ref{sec:intr}; thus, a SGFEM.
The generalization to higher-order elements has long been a challenging and unresolved problem.

\textbf{High-order GFEM:} To generalize linear GFEM to higher-order elements, the corresponding enrich space $V^h_{ENR}$ requires additional enrichment functions to achieve higher order local approximation. A common practice is to incorporate high-order polynomials into the enrichment functions. Take the standard GFEM~\eqref{eq:GFEM} for example, the set of enrichment functions at an enriched node expand from a singular $\{d(\boldsymbol{x})\}$ to multiple functions $\{d(\boldsymbol{x})x^iy^j, i+j < p\}$. When $p = 2$, $V_i = \{d(\boldsymbol{x}), d(\boldsymbol{x})x, d(\boldsymbol{x})y\}$. We abbreviate $\{d(\boldsymbol{x}), d(\boldsymbol{x})x, d(\boldsymbol{x})y \}$ as $d(\boldsymbol{x}) \{1, x, y\}$ to simplify the notation. The enrich space of GFEM of degree two is given by
\begin{equation}
    V^h_{ENR} =\text{span} \left\{\phi^2_i(\boldsymbol{x}) d(\boldsymbol{x})\{1, x, y\}, i \in \mathit{I}^{0}_{n} \right\}.
\end{equation}
Another option is to use $\phi^1_i$ as the PU function for high-order elements, where the enrichment space is defined as
\begin{equation}\label{eq:GFEM2}
    V^h_{ENR} =\text{span} \left\{\phi^1_i(\boldsymbol{x}) d(\boldsymbol{x})\{1, x, y\}, i \in \mathit{I}^{0}_{v} \right\}.
\end{equation}
This high-order extension approach can be applied to other GFEM, such as the aforementioned corrected GFEM~\eqref{eq:CorrGFEM} and SGFEM~\eqref{eq:SGFEM}.
However, these natural higher-order generalizations lack optimal convergence rates both in numerical experiments and error analysis.

\textbf{Quadratic corrected GFEM:}  Corrected GFEM with multiple enrichment terms is discussed in~\cite{cheng2010higher}. The enrich
space of degree two is
 \begin{equation}\label{eq:CGFEM2}
     V^h_{ENR} = \text{span} \left\{ \phi^1_i d(\boldsymbol{x})R(\boldsymbol{x})\{1, x, y, x^2, xy, y^2\}  ,\ i \in  \mathit{I}^{1}_{v} \right\}.
 \end{equation}
 Here, high-order polynomial terms are selected for improved accuracy. 
 However, this method still yields suboptimal results, and the issue of linear dependence between the two spaces remains a significant challenge.

\textbf{Quadratic SGFEM:} Stable GFEM of degree two based on \eqref{eq:SGFEM} was developed by Zhang and Babu\v{s}ka recently in~\cite{zhang2020stable}. The enriched space is
 \begin{equation}\label{eq:SGFEM2}
     V^h_{ENR} =\text{span} \left\{ H^1_i \left(\tilde{d}(\boldsymbol{x})\{1, x, y\} - \mathcal{I}_h (\tilde{d}\{1, x, y\})(\boldsymbol{x}) \right),\ i \in  \mathit{I}^{0}_{v} \right\},
 \end{equation}
 where $H^1_i$ are PU functions based on cubic Hermite interpolation polynomials and have the same DOF of $\phi^1_i$. $\tilde{d}$ is one-side distance function which equals $d$ on $\Omega_0$ and equals zero on $\Omega_1$. These modifications in the construction of the enrichment functions are all designed to reduce the linear dependence among the basis functions. The theoretical analysis of the optimal convergence for the \eqref{eq:SGFEM2} is also provided. To achieve a condition number growth rate comparable to standard FEM, local principal component analysis(LPCA)~\cite{abdi2010principal} method is used to preconditioning the stiffness matrix. We denote by $\mathbf{K}_{i}$ the $3 \times 3$ submatrix of the stiffness matrix corresponding to the enrichment functions $H^1_i \left(\tilde{d}\{1, x, y\} - \mathcal{I}_h \tilde{d}\{1, x, y\}) \right)$. Let its eigenvalues, sorted in descending order, be $\lambda_1 \geq \lambda_2 \geq \lambda_{3}$ and the corresponding percentages be given by $\xi_i :=\frac{\lambda_i}{\sum_{i=1}^{3} \lambda_i}$. Given a threshold $\xi$, principal components with values below this threshold will be removed. More details can be found in~\cite{zhang2020stable}. Numerical results demonstrate that the method achieves optimal convergence in the quadratic case. However, our numerical experiments will demonstrate that the straightforward extension of the method to high-order elements fails to produce optimal convergence results.

\subsection{HoSGFEM}\label{sec:hosgfem}
In this section, we present the main idea of high-order SGFEM for the interface problem \eqref{eq:pde}. As we have discussed in the last section, the key to the development of methods with desired properties (a), (b), and (c) as outlined in Section \ref{sec:intr}
is to define the enrichment space. 
%
The main challenge for the construction is how to define the space such that we can mathematically prove the optimal convergence rates, while reducing the linear dependence of the basis functions in both the spaces $V^h_{FEM}$ and $V^h_{ENR}$.
The linear independence restricts the resulting system from rapid growth in the condition number. With this in mind, we propose a two-step procedure to construct the enriched space, detailing the formulation of the basis functions $\psi_i$ and the corresponding enrichment space $V^h_{ENR}$.

\begin{figure}[ht]
    \centering
    \includegraphics[width=0.45\linewidth]{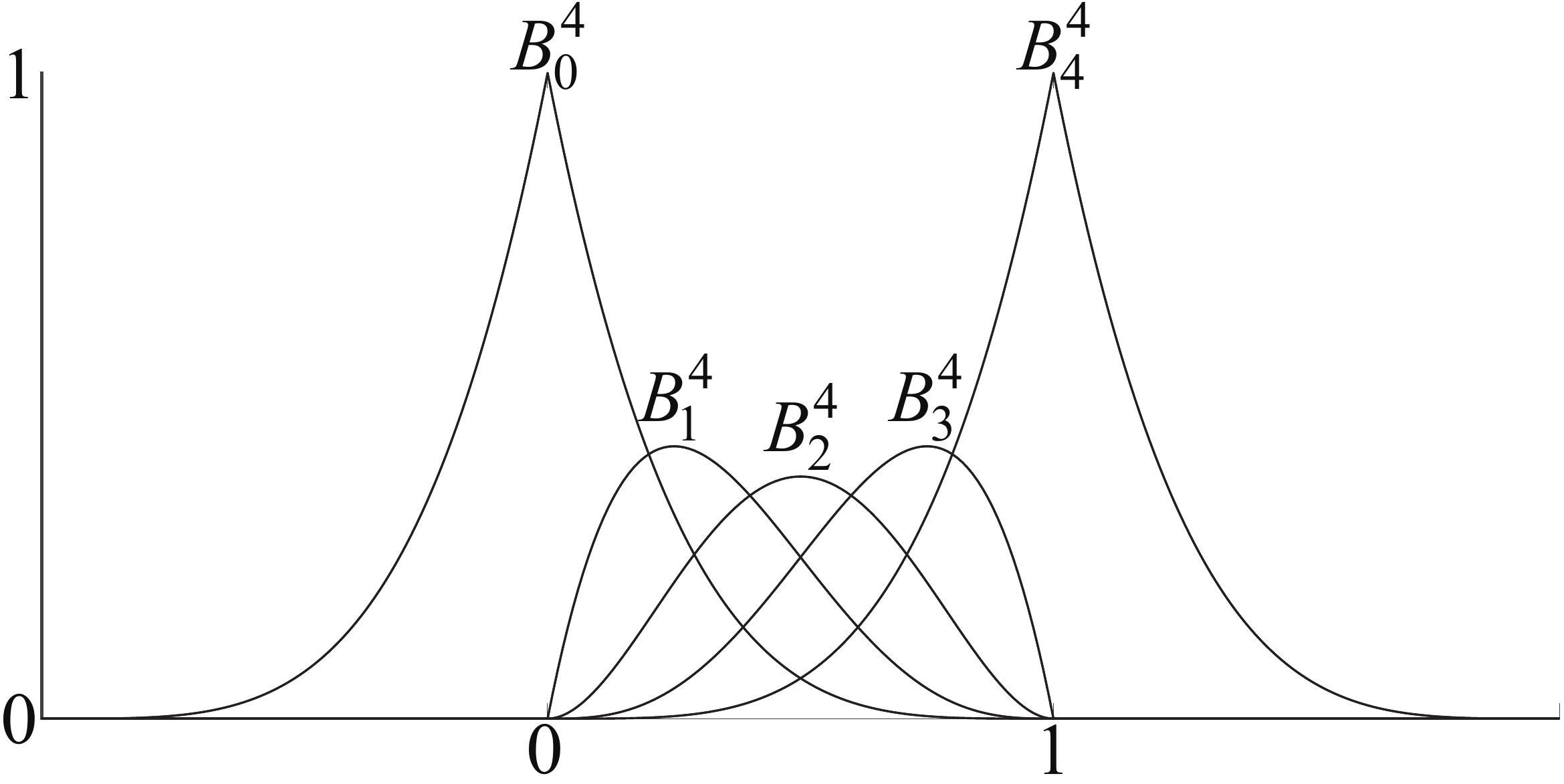}
    \caption{An example of the 4-th order Bernstein polynomials}
    \label{Fig:BernsteinPolynomials}
\end{figure}

\subsubsection{Construction of PU functions for an interface element}
It has been reported in \cite{zhang2014higher} that using linear basis functions $\phi^1_i$ for high-order GFEM may cause near-singularity in the stiffness matrix, whereas PU functions based on flat-top functions~\cite{oh2008piecewise, hong2013mesh} or higher-degree polynomials exhibit better stability. 
With these insights, we define high-order PU functions based on Bernstein polynomials for all interface elements. Specifically, for each interface element, we define one PU function.

We first consider the 1D definition of the Bernstein polynomials. Let $B^p_i(s)$ denote the $i$-th Bernstein polynomial of degree $p$ defined on the interval [0,1], which is
\begin{equation*}
    B^p_i(s) = \frac{p!}{i!(p-i)!} (1 - s)^{p - i} s^{i}, \ i = 0,\ldots, p.
\end{equation*}
It can be directly verified that $\sum_{i = 0}^p B^p_i(s) = 1$ for $s\in[0, 1]$. 
We refer to Figure~\ref{Fig:BernsteinPolynomials} for an example of quartic Bernstein polynomials. 
For the two-dimensional reference element $[0,1]\times[0,1]$, the bivariate polynomial $B^p_{i,j}(s, t)$ can be obtained by the tensor product of univariate polynomials in both directions, expressed as $B^p_{i,j}(s, t) := B^p_i(s) B^p_j(t)$. Bernstein functions on any quadrilateral element can be defined similarly via mapping the element to $[0,1]\times[0,1]$.

\begin{figure}[htbp]
    \centering
    \hspace{2cm}
    \subfloat[Nodes of Bernstein Polynomials in 2D]{\centering\includegraphics[width=0.2\linewidth]{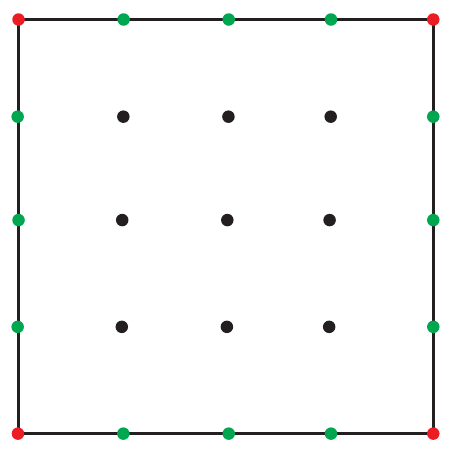}}
    \hfill
    \subfloat[Calculation of vertex coefficients]{\centering\includegraphics[width=0.2\linewidth]{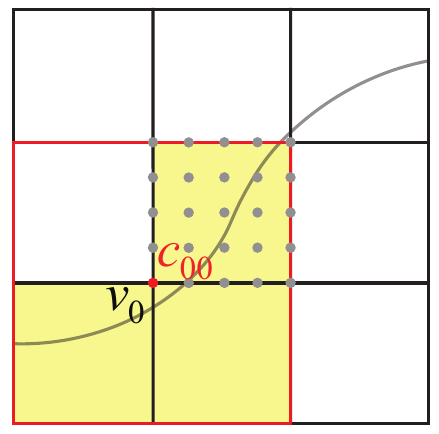}}
    \hfill
    \subfloat[Calculation of edge coefficients]{\centering\includegraphics[width=0.2\linewidth]{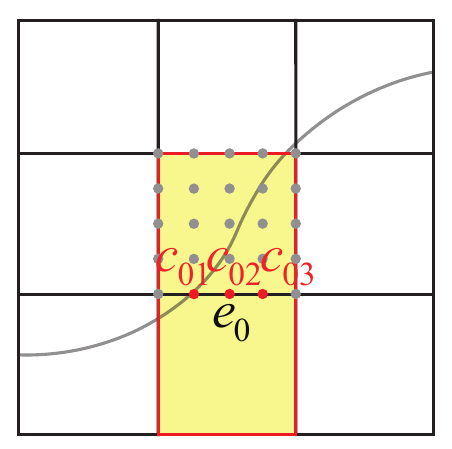}}
    \hspace{2cm}
    \caption{Construction of PU functions $\varphi_{k}(x, y)$.}\label{fig:computeCoef}
\end{figure}

Now we discuss how to define the PU function $\varphi^p_{k}(x, y)$ for the $k$-th interface element. We consider a generic interface element and omit the index $k$ for simplicity. 
Any bi-degree $p$-th order polynomial defined in the interface element can be represent as the linear combination of the Bernstein polynomials $B^p_{i,j}(x, y)$ with coefficients $c_{i, j}$, so suppose
$\varphi^p (x,y)$ has the form of
\begin{equation}\label{eq:PUfunction}
    \varphi^p(x, y) = \sum_{i = 0}^p \sum_{j = 0}^p c_{i,j} B^p_{i,j}(x, y),
\end{equation}
where the coefficients $c_{i,j}$ are to be determined. 
These coefficients can be categorized into three groups: (1) vertex coefficients $c_{0,0}, c_{p,0}, c_{0,p}, c_{p,p}$ which correspond to four vertices; (2) edge coefficients $c_{i,0}, c_{0,i}, c_{i,p}, c_{p,i}(i=1,\ldots,p-1)$ which correspond to four edges; and (3) element coefficients $c_{i,j}(i,j=1,\ldots,p-1)$, which correspond to the nodes inside the element. Figure \ref{fig:computeCoef}(a) shows the corresponding nodes using different colors for the associated coefficients $c_{i,j}$ for the example of bi-quartic element.
We give a simple procedure to compute these coefficients. 
\begin{itemize}
\item All the internal element coefficients are set to one, i.e., $c_{i,j} = 1$ for $i,j=1,\ldots,p-1$;

\item For the edge coefficients, if both adjacent elements are interface elements, then the coefficients are set to $\frac{1}{2}$, otherwise set them to one;

\item For the vertex coefficients, let $n_{v}$ denote the number of its adjacent elements that are interface elements; the coefficient is set to $\frac{1}{n_{v}}$.

\end{itemize}
For example, consider the vertex coefficient marked in red in Figure \ref{fig:computeCoef}(b). $n_{v}$ for the red vertex is $3$, thus $c_{0,0}=1/3.$
Similarly, for the edge coefficients marked in red in Figure \ref{fig:computeCoef}(c), the coefficients are $c_{0,1}=c_{0,2}=c_{0,3}=1/2.$

It can be verified that with these coefficients, $\varphi^p_{k}$ constitute a partition of unity property on the interface elements because
\begin{equation}
    \sum_{k \in I^0_{el}} \sum_{i,j=0}^pc^k_{ij}B^p_{ij} (x, y)
    = \sum_{i,j=0}^p \left(\sum_{k \in I^0_{el}} c^k_{ij}\right) B^p_{ij}(x, y) = \sum_{i,j=0}^pB^p_{ij} (x, y)= 1.
\end{equation}
To see this, we note that the Bernstein polynomials associated with the vertex and edge nodes have larger but still local support. 
We also remark that the function $\varphi^p$ can be analogously constructed using standard finite element basis functions (Lagrangian polynomials).

Figure~\ref{fig:PUfunctions} presents several PU function graphs on the circular interface. The newly constructed PU functions are different
from those existing linear FE basis functions $\phi^1_i$ or Hermite polynomial-based PU functions $H^1_i$~\eqref{eq:SGFEM2} in that each new PU function corresponds to one interface element rather than the vertex of interface elements. Furthermore, the polynomial degree of the PU function is the same as the FEM space. Numerical experiments demonstrate that the proposed PU function yields enhanced stability and robustness (corresponding to desired properties (b) and (c) discussed in Section \ref{sec:intr}).
\begin{figure}[htbp]
    \centering
    \includegraphics[width=0.7\linewidth]{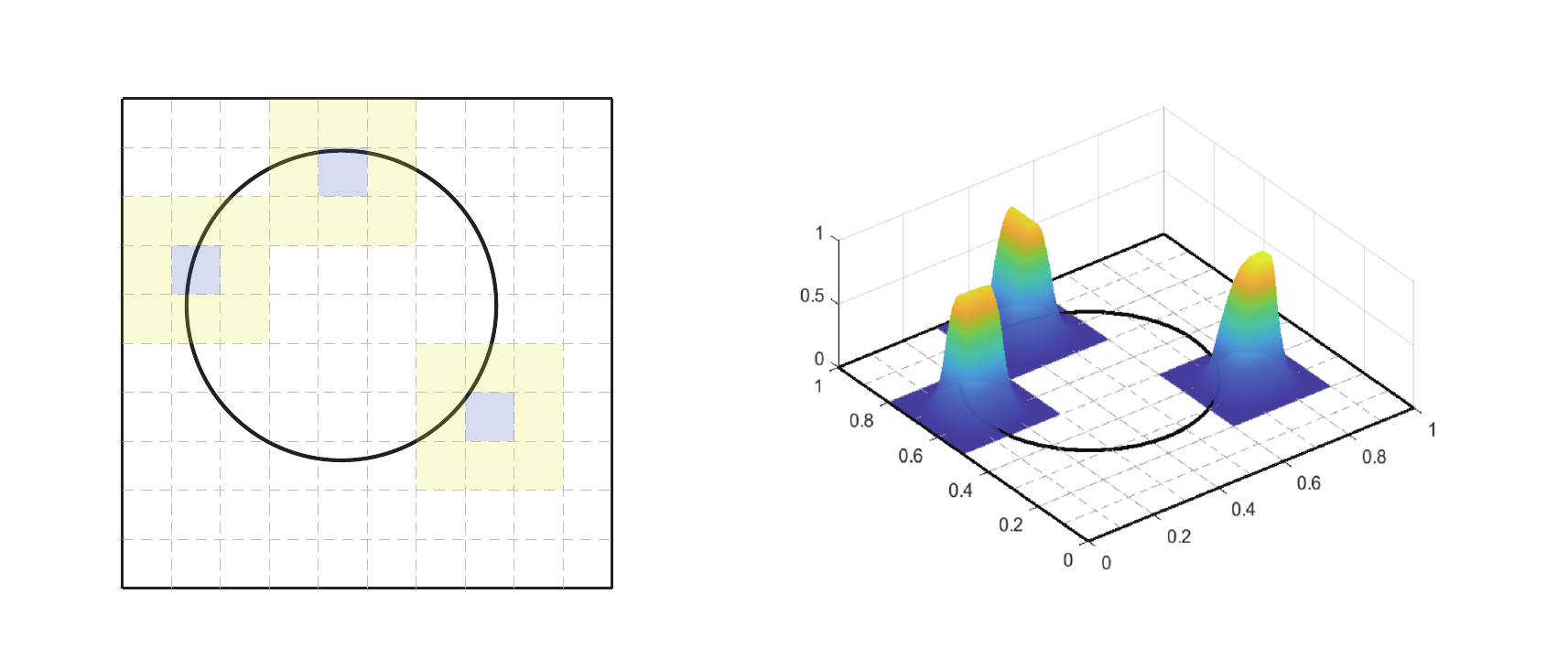}
    \caption{The left plot highlights several interface elements in light blue, and the right plot shows the corresponding PU functions.}
    \label{fig:PUfunctions}
\end{figure}

\subsubsection{Construction of the enrichment space}

In both constructions \eqref{eq:SGFEM} and~\eqref{eq:SGFEM2}, the distance function $d(\boldsymbol{x})$ and the interpolation operator $\mI$ have played an important role for optimal convergence and numerical stability as detailed in \cite{zhang2020stable, babuvska2012stable}. 
At each enriched node, the distance function is multiplied by a set of polynomials $x^iy^j$, $0 \leq i+j < p$. However, these auxiliary functions are global functions that limit the robustness of the numerical stability, especially for higher-order elements.  
With this in mind, we propose a shift for these functions for each interface element based on the centroid of the element. 
We find that this way of construction reduces the linear dependency among the functions, leading to better conditioning of the resulting system. 
For instance, suppose the centroid coordinates of the $k$-th interface element is $(x_k, y_k)$, then each PU function $\varphi^p_k$ multiplies the following polynomial set
\begin{equation}
\left \{(x - x_k)^i (y - y_k)^j, 0\leq i + j < p \right\} = \left\{1, \left(x - x_k\right) \ldots, \left(x - x_k\right) \left(y - y_k\right)^{p-2} ,\left(y - y_k\right)^{p-1} \right\}.
\end{equation}
Thus, the local enrichment function at each enriched node is $\varphi^p_k (d(x - x_k)^i (y - y_k)^j - \mathcal{I}^h(d(x - x_k)^i (y - y_k)^j))$.
Denote these functions to be $\{\xi_m\}_{m = 1, \ldots, p(p+1)/2}$, and then we perform Gram-Schmidt orthogonalization process (see Algorithm \ref{alg:gs}) to compute a set of new basis functions $\{\psi_m\}_{m = 1, \ldots, p(p+1)/2}$ for each interface element.
\begin{algorithm}
\caption{Gram-Schmidt Orthogonalization}
\label{alg:gs}
\begin{algorithmic}[1]
\Require $\{\xi_m\}_{m = 1,\ldots, p(p+1)/2}$.
\Ensure $\{\psi_m\}_{m = 1,\ldots, p(p+1)/2}$
\State $\tilde{\psi}_1 \gets \xi_1$
\For{$l = 2$ to $p(p+1)/2$}
\State $\tilde{\psi}_l \gets \xi_l$
\For{$j = 1$ to $l-1$}
\State $\tilde{\psi}_l \gets \tilde{\psi}_l - \frac{\left( \xi_l,\tilde{\psi}_j \right)}{ \left(\tilde{\psi}_j, \tilde{\psi}_j \right)}
\psi_j$, where $\left( f, g \right) = \int_{\Omega} fg dx $
\EndFor
\EndFor
\For{$i = 1$ to $p(p+1)/2$}
\State $\psi_i \gets \frac{\tilde{\psi}_i}{\Vert\tilde{\psi}_i\Vert} $
\EndFor
\end{algorithmic}
\end{algorithm}

The procedure is performed locally among $(p+1)p/2$ functions at each enriched element, where the integration is evaluated within the local support of the basis function associated with this element. For example, the yellow regions corresponding to the blue interface element in Figure~\ref{fig:PUfunctions} are the regions for the integration in the orthogonalization. 
This localized approach requires minimal additional computational effort. 
We remark that this Gram-Schmidt orthogonalization is not needed for $p\le4$. 
We observe in numerical simulation that this orthogonalization improves the conditioning of the system when $p\ge5.$
Let $\mathcal{L}$ denote the aforementioned orthogonalization operation, i.e., $\{\psi_m \} = \mathcal{L}\{\xi_m\}$, then the final enrichment space is defined as
\begin{equation}\label{eq:impSGFEMp}
    V^h_{ENR} =\text{span}\left\{ \mathcal{L} \left(\varphi^p_k d (x-x_k)^i(y - y_k)^j - \varphi^p_k
    \mathcal{I}_h\left( d (x-x_k)^i(y - y_k)^j \right) \right),\ k\in I^0_{el},\ 0 \leq i + j< p \right\}.
\end{equation}

With this enrichment space, the proposed HoSGFEM is of formulation \eqref{eq:GalApproProm} with the solution and test space defined in \eqref{eq:vhgfem}.
In Section \ref{sec:ea} below, we establish that the proposed method results in an optimally convergent approximate solution to the interface problem \eqref{eq:pde}, implying the desired property (a) on convergence. 
In Section \ref{sec:num} for numerical experiments, we present various numerical examples to demonstrate the desired property (b) numerical stability and (c) robustness. 
Thus, the proposed method is indeed a stable GFEM with a unified generalization to high-order elements.

\section{Error Analysis}\label{sec:ea}
In this section, we establish the optimal error convergence rates for the proposed HoSGFEM with the construction~\eqref{eq:impSGFEMp} for the interface problem~\eqref{eq:pde}. 
Unless otherwise specified, the notation in the following proof remains the same as before. 
For simplicity, we use $C$ to denote a generic constant that is independent of the element size $h$.

Assume that solution $u$ of problem~\eqref{eq:pde} is
\begin{equation}\label{eq:extendu}
     u(\boldsymbol{x})= \begin{cases}u_0(\boldsymbol{x}), & \boldsymbol{x} \in \bar{\Omega}_0,
     \\ u_1(\boldsymbol{x}), & \boldsymbol{x} \in \bar{\Omega}_1. \end{cases}
\end{equation}
Smoothly extending $u_i$ from $\Omega_i$ to entire domain $\Omega$, one gets $\tilde{u}_i \in H^{p + 1}(\Omega) \cap W^{p, \infty}(\Omega)$ satisfy
\begin{equation}\label{eq:propertyU}
    \tilde{u}_i(\boldsymbol{x}) = u_i(\boldsymbol{x})\ (\boldsymbol{x} \in \Omega_i), \quad \Vert \tilde{u}_i \Vert_{H^{p+1}(\Omega)} \leq C \Vert u_i \Vert_{H^{p+1}(\Omega_i)}\ \text{and}\  \Vert \tilde{u}_i \Vert_{W^{p, \infty}(\Omega)} \leq C \Vert u_i \Vert_{W^{p, \infty}(\Omega_i)}.
\end{equation}
Define $\tilde{u}$ as
\begin{equation}
    \tilde{u}(\boldsymbol{x}) = \begin{cases} \tilde{u}_0(\boldsymbol{x}) - \tilde{u}_1(\boldsymbol{x}), & \boldsymbol{x} \in \bar{\Omega}_0,\\0, & \boldsymbol{x} \in \bar{\Omega}_1. \end{cases}
\end{equation}
and define one-side distance function $\tilde{d}$
\begin{equation}
    \tilde{d}(\boldsymbol{x}) = \begin{cases} {d}(\boldsymbol{x}), & \boldsymbol{x} \in \bar{\Omega}_0,\\0, & \boldsymbol{x} \in \bar{\Omega}_1. \end{cases}
\end{equation}

Define $w_h := \sum_{i \in I^0_{el}} \phi_i (\mathcal{I} - \mathcal{I}_h)(\tilde{d} \eta_i)$, where $\mathcal{I}$ is the identity operator and $\eta_i$ is the polynomial of degree less than $p$. It is clear that $w_h \in V^h_{ENR}$. We have
\begin{equation}\label{eq:VenrApproach}
    \begin{split}
    \tilde{u} - \mathcal{I}_h \tilde{u} - w_h &= (\mathcal{I} - \mathcal{I}_h)\tilde{u} - \sum_{i \in I^0_{el}} \phi_i (\mathcal{I} - \mathcal{I}_h)(\tilde{d} \eta_i)\\
    & = (1 - \sum_{i \in I^0_{el}} \phi_i) (\mathcal{I} - \mathcal{I}_h)\tilde{u} +  \sum_{i \in I^0_{el}}\phi_i (\mathcal{I} - \mathcal{I}_h)\tilde{u} - \sum_{i \in I^0_{el}} \phi_i (\mathcal{I} - \mathcal{I}_h)(\tilde{d} \eta_i)\\
    & = (1 - \sum_{i \in I^0_{el}} \phi_i) (\mathcal{I} - \mathcal{I}_h)\tilde{u} +  \sum_{i \in I^0_{el}}\phi_i (\tilde{u} - \tilde{d} \eta_i) - \sum_{i \in I^0_{el}}\phi_i \mathcal{I}_h(\tilde{u} - \tilde{d} \eta_i).
    \end{split}
\end{equation}
Let $\bar{\phi} := (1 - \sum_{i \in I^0_{el}} \phi_i)$. Given that $\sum_{i \in I^0_{el}} \phi_i = 1$ on interface elements, it follows that $\text{supp}(\bar{\phi}) \cap \Gamma = \emptyset$. Therefore, the error in the first term of the final line follows directly from standard FEM interpolation estimates for smooth functions. Now we establish the estimate for the $\tilde{u} - \tilde{d}\eta_i$ in Lemma \ref{lemma1}.
\begin{lemma}\label{lemma1}
    For an interface element $e_i$, there exists a polynomial $\eta_i$ of degree $(p - 1)$ such that
    \begin{equation}\label{eq:AppendLemma}
        \Vert \tilde{u}  - \tilde{d}\eta_i \Vert^2_{H^l(\omega_i)} \leq C h^{2(p + 1-l)} \Vert \tilde{u}_0 - \tilde{u}_1 \Vert^2_{H^{p + 1}(\omega_i)}
    \end{equation}
    \begin{equation}
        \Vert \mathcal{I}_h (\tilde{u}  - \tilde{d}\eta_i) \Vert^2_{H^l(\omega_i)} \leq C h^{2(p + 1 - l)}\Vert \tilde{u}_0- \tilde{u}_1 \Vert^2_{H^{p + 1}(\omega_i)}
    \end{equation}
where $\omega_i$ is the support of PU function $\phi_i$ corrosponding to interface element $e_i$.
\end{lemma}
\begin{proof}
    Given that $\tilde{u}$ and $\tilde{d}$ are zero on $\Omega_1$, we immediately get
    \begin{equation}
        \Vert \tilde{u} - \tilde{d} \eta_i \Vert^2_{H^l(\omega_i \cap \Omega_1)} = 0.
    \end{equation}
    Now we estimate the part $\Vert \tilde{u} - \tilde{d} \eta_i \Vert^2_{H^l(\omega_i \cap \Omega_0)}$. Let $\boldsymbol{n}$ and $\boldsymbol{\tau}$ be the unit normal vector and unit tangent vector at a point $\boldsymbol{y}_i \in \Gamma \cap \omega_i$.
    \begin{figure}[htbp]
    \centering
    \includegraphics[width=0.25\linewidth]{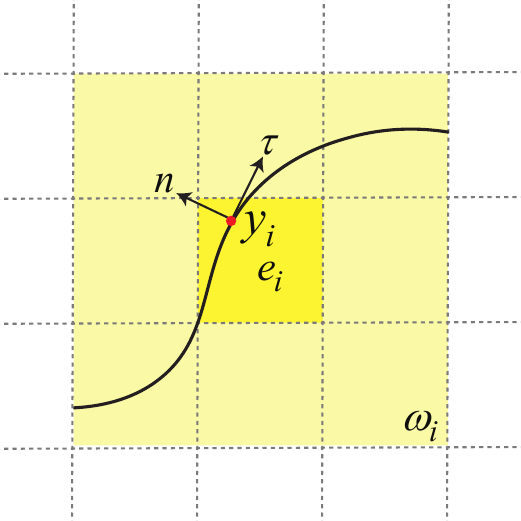}
    \caption{Illustration of local coordinates}
    \label{Fig:local coordinates}
    \end{figure}
    
    Noticing that $\boldsymbol{n}$ and $\boldsymbol{\tau}$ form a local orthogonal coordinate system, we consider the Taylor polynomial at the point $\boldsymbol{y}_i$ using these orthogonal vectors as the coordinate system. To simplify notation, let $\boldsymbol{\alpha} = (\alpha_0, \alpha_1)$ be a multi-index, denote $\boldsymbol{\alpha}! := a_0! a_1!$, $|\boldsymbol{\alpha}| := a_0 + a_1$, $ (\boldsymbol{x} - \boldsymbol{y}_i)^{\boldsymbol{\alpha}} := [(\boldsymbol{x} - \boldsymbol{y}_i) \cdot \boldsymbol{n}]^{a_0} [(\boldsymbol{x} - \boldsymbol{y}_i) \cdot \boldsymbol{\tau}]^{a_1}$ and
    \begin{equation}
        D^{\boldsymbol{\alpha}} := \underbrace{ \frac{\partial}{\partial \boldsymbol{\tau}} \cdots \frac{\partial}{\partial \boldsymbol{\tau}}}_{a_1} \underbrace{ \frac{\partial}{\partial \boldsymbol{n}} \cdots \frac{\partial}{\partial \boldsymbol{n}}}_{a_0}.
    \end{equation}
    Taylor polynomial of order $p$ for function $\tilde{u}$ at $\boldsymbol{y}_i$ is
    \begin{equation}
        T^p_i \tilde{u}(\boldsymbol{x}) = \sum_{|\boldsymbol{\alpha}| < p} \frac{1}{\boldsymbol{\alpha}!} D^{\boldsymbol{\alpha}}\tilde{u}(\boldsymbol{y}_i) (\boldsymbol{x} - \boldsymbol{y}_i)^{\mathbf{\boldsymbol{\alpha}}}.
    \end{equation}
    For $\boldsymbol{x} \in \omega_i \cap \Omega_0$, there holds
    \begin{equation}
        \begin{aligned}
             \left\Vert \tilde{u} -\tilde{d}\eta_i \right\Vert^2_{H^l(\omega_i \cap \Omega_0)} & =
            \left\Vert \tilde{u} - T^{p+1}_i\tilde{u} +T^{p+1}_i\tilde{u} - (T^{p+1}_i\tilde{d})\eta_i + (T^{p+1}_i\tilde{d})\eta_i -\tilde{d}\eta_i
            \right\Vert^2_{H^l(\omega_i \cap \Omega_0)}\\
            & \leq 3\left\Vert \tilde{u} - T^{p+1}_i\tilde{u} \right\Vert^2_{H^l(\omega_i \cap \Omega_0)} + 3\left\Vert T^{p+1}_i\tilde{u} - (T^{p+1}_i\tilde{d})\eta_i \right\Vert^2_{H^l(\omega_i \cap \Omega_0)} + 3\left\Vert (T^{p+1}_i\tilde{d})\eta_i -\tilde{d}\eta_i \right\Vert^2_{H^l(\omega_i \cap \Omega_0)}.
        \end{aligned}
    \end{equation}
    The estimate of the first and the third term of the equation above is
    \begin{equation}
        \begin{aligned}
         \left\Vert \tilde{u} - T^{p+1}_i\tilde{u} \right\Vert^2_{H^l(\omega_i \cap \Omega_0)}  \leq &\ \left\Vert \tilde{u}_0 - \tilde{u}_1 - T^{p+1}_i (\tilde{u}_0 - \tilde{u}_1) \right\Vert^2_{H^l(\omega_i)} \leq C h^{2(p +1-l)}\left\Vert \tilde{u}_0 - \tilde{u}_1 \right\Vert^2_{H^{p+1}(\omega_i)}, \\
            \left\Vert (T^{p+1}_i\tilde{d})\eta_i -\tilde{d}\eta_i \right\Vert^2_{H^l(\omega_i \cap \Omega_0)} \leq &\ C\left\Vert (T^{p+1}_i\tilde{d})\eta_i -\tilde{d}\eta_i \right\Vert^2_{W^{l, \infty}(\omega_i \cap \Omega_0)}|\omega_i \cap \Omega_0|\\
            \leq &\ C\left\Vert T^{p+1}_i\tilde{d} -\tilde{d} \right\Vert^2_{W^{l, \infty}(\omega_i \cap \Omega_0)} \left\Vert \eta_i \right\Vert^2_{W^{l, \infty}(\omega_i \cap \Omega_0)} |\omega_i|\\
            \leq &\ C h^{2 (p+2 - l)} \left\Vert \eta_i \right\Vert^2_{W^{l, \infty}},
        \end{aligned}
    \end{equation}
where $|\omega_i|$ is the area of $\omega_i$.

Now we will estimate the second term
\begin{equation}\label{eq:ComputeEps}
    \epsilon = T^{p+1}_i \tilde{u} - (T^{p+1}_i \tilde{d})\eta_i.
\end{equation}

For simplicity, denote by $u_{\boldsymbol{\alpha}}$ and $d_{\boldsymbol{\alpha}}$ the coefficients of the polynomial $(\boldsymbol{x} - \boldsymbol{y}_i)^{\boldsymbol{\alpha}}$ of $T^{p+1}_i \tilde{u}$ and $ T^{p+1}_i \tilde{d}$. 
We write the  polynomial $\eta_i$ of degree $p- 1$ in the form
$$\eta_i = \sum_{|\boldsymbol{\gamma}| < p} c_{\boldsymbol{\gamma}} (\boldsymbol{x} - \boldsymbol{y}_i)^{\boldsymbol{\gamma}},$$
where $c_{\boldsymbol{\gamma}}$ is the coefficients to be determined.

By expanding the second term in equation~\eqref{eq:ComputeEps}, we get
\begin{align}
    \epsilon = &\sum_{|\boldsymbol{\alpha}| < p+1} u_{\boldsymbol{\alpha}} (\boldsymbol{x} - \boldsymbol{y}_i)^{\boldsymbol{\alpha}} - \left(\sum_{|\boldsymbol{\beta}| < p + 1} d_{\boldsymbol{\beta}} (\boldsymbol{x} - \boldsymbol{y}_i)^{\boldsymbol{\beta}}\right) \left( \sum_{|\boldsymbol{\gamma}| < p } c _{\boldsymbol{\gamma}} (\boldsymbol{x} - \boldsymbol{y}_i)^{\boldsymbol{\gamma}} \right)\\
    = & \sum_{|\boldsymbol{\alpha}| < 2p} \epsilon_{\boldsymbol{\alpha}} (\boldsymbol{x} - \boldsymbol{y}_i)^{\boldsymbol{\alpha}}.
\end{align}
We consider to determine the $c_{\boldsymbol{\gamma}}$ to make the $\epsilon_{\boldsymbol{\alpha}} = 0$ for $|\boldsymbol{\alpha}| < p + 1$. Let $\boldsymbol{\alpha} = (i,j), \boldsymbol{\gamma} = (m,n)$, the coefficient $\epsilon_{i, j}$ is computed as
\begin{align}\label{eq: proof-coe}
    \epsilon_{i,j} = u_{i,j} - \sum_{\substack{0 \leq m \leq i \\ 0 \leq n \leq j}} d_{i - m,j - n}c_{m, n}, 0 \leq i + j \leq p.
\end{align}
Since $\tilde{u}$ and $\tilde{d}$ are zero on the interface $\Gamma$, their corresponding tangential derivative coefficients $u_{0,j}$ and $d_{0,j}$ are zero for $j \leq p$.  Furthermore, according to the definition of distance function, we have $d_{1,0}= 1$, which has been proved in~\cite{zhang2020stable,zhu2020stable}. Then we immediately have $\epsilon_{0, j} = 0$ for  $j \leq p$. We only need to consider the coefficients $\epsilon_{i, j} (i + j \leq p$, $i \neq 0)$. The number of these coefficients is $(1+p)p/2$, which is equal to the number of unknowns $c_{m, n}(m+n < p)$, thus, there must be a set of $c_{m, n}$ satisfying the conditions. 
In fact, one can determine $c_{m, n}$ explicitly by induction. Firstly, when $m + n = 0$, we have
\begin{equation}
    \epsilon_{1, 0} = u_{1, 0} - d_{1, 0} c_{0, 0} - d_{0, 0}c_{1, 0} = 0,
\end{equation}
which means $c_{0, 0} = u_{1, 0}$. Assume that $c_{m,n} (m + n < l)$ has been determined, then  $c_{m,n}(m + n = l)$ can be determined by the following equation
\begin{equation}\label{eq:ComputeCmn}
\begin{split}
    c_{m, n} & = c_{m, n} d_{1, 0} \\
    & = u_{m + 1, n} - c_{m+1, n} d_{0, 0} - c_{m + 1, n - 1} d_{0, 1} - \sum_{i+j < l} c_{i, j} d_{m - i + 1, n - j}\\
    & = u_{m + 1, n} - \sum_{i + j < l} c_{i,j} d_{m - i + 1, n - j}.
\end{split}
\end{equation}
Therefore, it is easy to verify that each $c_{m, n}$ is uniquely determined to satisfy condition $\epsilon_{m + 1, n} = 0$.

After determining the coefficients $c_{m, n}$, equation~\eqref{eq:ComputeEps} can be express as
\begin{equation}
    \epsilon = T^{p + 1}_i \tilde{u} - \left(T^{p + 1}_i \tilde{d}\right)\eta_i = \sum_{|\boldsymbol{\alpha}| > p} \epsilon_{\boldsymbol{\alpha}} (\boldsymbol{x} - \boldsymbol{y}_i)^{\boldsymbol{\alpha}},
\end{equation}
and $\epsilon_{\boldsymbol{\alpha}}$ is the function of the partial derivatives $\tilde{u}$ and $\tilde{d}$, which does not depend on
element size $h$. 
As a consequence, we have
\begin{equation}
    \left\Vert T^{p+1}_i \tilde{u} - (T^{p+1}_i \tilde{d})\eta_i \right\Vert^2_{H^l (\omega_i \cap \Omega_0)} \leq C \left\Vert \sum_{|\boldsymbol{\alpha}| > p} \epsilon_{\boldsymbol{\alpha}} (\boldsymbol{x} - \boldsymbol{y}_i)^{\boldsymbol{\alpha}}\right\Vert^2_{W^{l, \infty} (\omega_i \cap \Omega_0)} \left|\omega_i \cap \Omega_0\right| \leq C h^{2 (p + 2 - l)}.
\end{equation}
Finally, we get the estimate for $\left\Vert \tilde{u}  - \tilde{d}\eta_i \right\Vert^2_{H^l(\omega_i)}$
\begin{equation}
    \begin{split}
         \left\Vert \tilde{u}  - \tilde{d}\eta_i \right\Vert^2_{H^l(\omega_i)}
         \leq & C h^{2 (p + 1 - l)}\left\Vert \tilde{u}_0 - \tilde{u}_1 \right\Vert^2_{H^{p+ 1}(\omega_i)} + C h^{2 (p + 2 - l)} + C h^{2 (p + 2 - l)} \left\Vert \eta_i \right\Vert^2_{W^{1, \infty}(\omega_i)}\\
         \leq & C h^{2 (p + 1 - l)} \left\Vert \tilde{u}_0 - \tilde{u}_1  \right\Vert^2_{H^{p+1}(\omega_i)}.
    \end{split}
\end{equation}
For the second term, we have
\begin{equation}\label{eq:ComputeItpt}
    \begin{split}
    \mathcal{I}_h (\tilde{u}  - \tilde{d}\eta_i) = & \sum_{\boldsymbol{x}_j \in \omega_i } \left(\tilde{u} (\boldsymbol{x}_j) - \tilde{d}(\boldsymbol{x}_j) \eta_i(\boldsymbol{x}_j)\right) \phi^p_j \\
    = & \sum_{\boldsymbol{x}_j \in \omega_i } \left(\tilde{u} (\boldsymbol{x}_j) - (T^{p+1}_i \tilde{u})(\boldsymbol{x}_j)\right)\phi^p_j +
    \sum_{x_j \in \omega_i } \left((T^{p+1}_i \tilde{u})(\boldsymbol{x}_j )- \tilde{d}(\boldsymbol{x}_j) \eta_i(\boldsymbol{x}_j)\right) \phi^p_j,
    \end{split}
\end{equation}
where $\boldsymbol{x}_j$ are FE nodes. For the coefficient of $ \phi^p_j$, we have
\begin{equation}\label{eq:ComputeItptC0}
    \left|(\tilde{u} (\boldsymbol{x}_j) - (T^{p+1}_i \tilde{u})(\boldsymbol{x}_j)\right| \leq \left\Vert \tilde{u}- T^{p+1}_i \tilde{u}
    \right\Vert_{L^{\infty}(\omega_i)} \leq C h^{p+1} \left\Vert \tilde{u}_0 - \tilde{u}_1\right\Vert_{H^{p + 1}(\omega_i)},
\end{equation}
\begin{equation}\label{eq:ComputeItptC1}
    \left|(T^{p+1}_i \tilde{u})(\boldsymbol{x}_j )- \tilde{d}(\boldsymbol{x}_j) \eta_i(\boldsymbol{x}_j)\right| \leq
    \left\Vert T^{p+1}_i \tilde{u} - (T^{p+1}_i\tilde{d}) \eta_i + (T^{p+1}_i\tilde{d}) \eta_i  - \tilde{d} \eta_i\right\Vert_{L^{\infty}(\omega_i)} \leq C h^{p + 1}.
\end{equation}
Combining equations~\eqref{eq:ComputeItpt},~\eqref{eq:ComputeItptC0} and ~\eqref{eq:ComputeItptC1}, we obtain
\begin{equation}
    \begin{split}
        \left\Vert \mathcal{I}_h \left(\tilde{u}  - \tilde{d}\eta_i\right) \right\Vert^2_{H^l(\omega_i)} & \leq C \left[ \max_{\boldsymbol{x}_j \in \omega_i}{
        \left|\left(\tilde{u} - T^{p+1}_i \tilde{u}\right)(\boldsymbol{x}_j)\right|^2}\right] \sum_{\boldsymbol{x}_j \in \omega_i}
        \left\Vert \phi^p_j \right\Vert^2_{H^{l}(\omega_i)}\nonumber\\
        & +  C \left[ \max_{\boldsymbol{x}_j \in \omega_i } \left|T^{p+1}_i \left(\tilde{u}- \tilde{d}\eta_i\right)(\boldsymbol{x}_j)\right|^2 \right]
        \sum_{\boldsymbol{x}_j \in \omega_i} \left\Vert \phi^p_j \right\Vert^2_{H^{l}(\omega_i)}\nonumber\\
        & \leq C h^{2 (p + 1 - l)} \left\Vert \tilde{u}_0 - \tilde{u}_1 \right\Vert^2_{H^{p + 1}(\omega_i)},
    \end{split}
\end{equation}
which completes the proof.
\end{proof}

\begin{remark}
We give an example for determining the coefficient when $p = 3$. If $m + n = 1$, there exist two constraints 
\begin{equation}
    \begin{split}
        \epsilon_{2, 0} & = u_{2, 0} - d_{2, 0} c_{0, 0} - d_{1, 0} c_{1,0} = 0,\\
        \epsilon_{1, 1} & = u_{1, 1} - d_{1, 1} c_{0, 0} - d_{1, 0} c_{0,1} = 0.
    \end{split}
\end{equation}
So $c_{1, 0}$ and $c_{0, 1}$ can be computed as
\begin{equation}
    \begin{split}
        c_{1,0} &= u_{2, 0} - d_{2, 0} c_{0, 0}, \\
        c_{0,1} &= u_{1, 1} - d_{1, 1} c_{0, 0}.
    \end{split}
\end{equation}
If $m + n = 2$, there exist three constraints, 
\begin{equation}
    \begin{split}
        \epsilon_{3, 0} & = u_{3, 0} - d_{3, 0} c_{0, 0} - d_{2, 0} c_{1,0} - d_{1, 0} c_{2, 0}= 0,\\
        \epsilon_{2, 1} & = u_{2, 1} - d_{2, 1} c_{0, 0} - d_{1, 1} c_{1,0} - d_{2, 0} c_{0, 1} - d_{1, 0} c_{1, 1}= 0,\\
        \epsilon_{1, 2} & = u_{1, 2} - d_{1, 2} c_{0, 0} - d_{0, 2} c_{1,0} - d_{1, 1} c_{0, 1} - d_{1, 0} c_{0, 2}= 0.
    \end{split}
\end{equation}
Thus, $c_{2, 0}$, $c_{1, 1}$ and $c_{0, 2}$ can be solved as follows
\begin{equation}
    \begin{split}
        c_{2, 0} & = u_{3, 0} - d_{3, 0} c_{0, 0} - d_{2, 0} c_{1,0},\\
        c_{1, 1} & = u_{2, 1} - d_{2, 1} c_{0, 0} - d_{1, 1} c_{1,0} - d_{2, 0} c_{0, 1},\\
        c_{0, 2} & = u_{1, 2} - d_{1, 2} c_{0, 0} - d_{0, 2} c_{1,0} - d_{1, 1} c_{0, 1}.
    \end{split}
\end{equation}
\end{remark}

With the estimation of $\tilde{u} - \tilde{d}\eta_i $, we can derive the estimation of equation~\eqref{eq:VenrApproach}, which is proved in the following lemma.
\begin{lemma}\label{lemma2}Suppose $\eta_i$ is the polynomials that are computed in Lemma \ref{lemma1} and let $w_h := \sum_{i \in I^0_{el}} \phi_i (\mathcal{I} - \mathcal{I}_h)(\tilde{d}  \eta_i)$, then we have
    \begin{equation}
        \left| \tilde{u} - \mathcal{I} \tilde{u} - w_h \right|_{H^1(\Omega)} \leq C  h^{p}\left\Vert \tilde{u}_0 - \tilde{u}_1 \right\Vert_{H^{p+1}(\Omega)}.
    \end{equation}
\end{lemma}
\begin{proof}
Using equation~\eqref{eq:VenrApproach}, we compute
    \begin{equation}\label{eq:three}
        \begin{split}
            \left| \tilde{u} - \mathcal{I} \tilde{u} - w_h \right|^2_{H^1(\Omega)}
            & \leq \left|\bar{\phi} (\mathcal{I} - \mathcal{I}_h)\tilde{u} +  \sum_{i \in I^0_{el}}\phi_i (\tilde{u} - \tilde{d} \eta_i) - \sum_{i \in I^0_{el}}\phi_i \mathcal{I}_h(\tilde{u} - \tilde{d} \eta_i) \right|^2_{H^1(\Omega)}\\
            & \leq C \left|\bar{\phi} (\mathcal{I} - \mathcal{I}_h)\tilde{u} \right|^2_{H^1(\Omega)} + C \sum_{i \in I^0_{el}}
            \left| \phi_i (\tilde{u} - \tilde{d} \eta_i)\right|^2_{H^1(\Omega)} + C \sum_{i \in I^0_{el}} \left|\phi_i \mathcal{I}_h(\tilde{u} - \tilde{d} \eta_i) \right|^2_{H^1(\Omega)}.
        \end{split}
    \end{equation}
First, the support of $\bar{\phi}$ consists of all non-interface elements ($ I_{el}\backslash I^0_{el} $) and equals 1 on elements
outside the 1-ring neighborhood of interface elements($ I_{el}\backslash I^1_{el} $), so we have the estimation for the first term in equation~\eqref{eq:three},
    \begin{equation}
        \begin{split}
            \left|\bar{\phi} (\mathcal{I} - \mathcal{I}_h)\tilde{u} \right|^2_{H^1(\Omega)}
            & = \sum_{i \in I_{el} \backslash I^0_{el}} \left|\bar{\phi} (\mathcal{I} - \mathcal{I}_h)\tilde{u} \right|^2_{H^1(e_i)} \\
            & \leq C  \sum_{i \in I_{el} \backslash I^0_{el}} \left\Vert \bar{\phi}\right\Vert^2_{L^{\infty}(e_i)} \left|(\mathcal{I} - \mathcal{I}_h)\tilde{u}\right|^2_{H^1(e_i)}
            + C  \sum_{i \in I_{el} \backslash I^0_{el}} \left\Vert \nabla \bar{\phi}\right\Vert^2_{L^{\infty}(e_i)}
            \left\Vert (\mathcal{I} - \mathcal{I}_h)\tilde{u} \right\Vert^2_{L^2(e_i)}\\
            & \leq C  \sum_{i \in I_{el} \backslash I^0_{el}} h^{2p} \left|\tilde{u}_0 - \tilde{u}_1\right|^2_{H^{p+1}(e_i)}
            + C  \sum_{i \in I_{el} \backslash I^0_{el}} h^{-2} h^{2(p + 1)} \left|\tilde{u}_0 - \tilde{u}_1\right|^2_{H^{p+1}(e_i)}\\
            & \leq C h^{2 p} \left|\tilde{u}_0 - \tilde{u}_1\right|^2_{H^{p+1}(\Omega)}.
        \end{split}
    \end{equation}
    Now we can estimate of the second term of equation~\eqref{eq:three} using Lemma \ref{lemma1},
    \begin{equation}
        \begin{split}
             \sum_{i \in I^0_{el}} \left| \phi_i (\tilde{u} - \tilde{d} \eta_i)\right|^2_{H^1(\omega_i)}
            &\leq C \sum_{i \in I^0_{el}}  \left\Vert \phi_i \right\Vert^2_{L^{\infty}(\omega_i)} \left|\tilde{u} - \tilde{d} \eta_i\right|^2_{H^1(\omega_i)} +
             C  \sum_{i \in I^0_{el}} \left\Vert \nabla \phi_i \right\Vert^2_{L^{\infty}(\omega_i)} \left\Vert \tilde{u} - \tilde{d} \eta_i \right\Vert^2_{L^{2}(\omega_i)} \\
            & \leq C \sum_{i \in I^0_{el}}  \left|\tilde{u} - \tilde{d}\eta_i\right|^2_{H^1(\omega_i)} + C \sum_{i \in I^0_{el}}  h^{-2}
            \left\Vert \tilde{u} - \tilde{d} \eta_i \right\Vert^2_{L^{2}(\omega_i)} \\
            & \leq C \sum_{i \in I^0_{el}}  h^{2 p} \left\Vert \tilde{u}_0 - \tilde{u}_1 \right\Vert^2_{H^{p+1}(\omega_i)}\\
            & \leq C h^{2 p} \left\Vert \tilde{u}_0 - \tilde{u}_1 \right\Vert^2_{H^{p+1}(\Omega)}.
        \end{split}
    \end{equation}
    Similarly, we have
        \begin{equation}
             \sum_{i \in I^0_{el}} \left|\phi_i \mathcal{I}_h (\tilde{u} - \tilde{d} \eta_i)\right|^2_{H^1(\omega_i)}
             \leq C h^{2 p} \left\Vert \tilde{u}_0 - \tilde{u}_1 \right\Vert^2_{H^{p+1}(\Omega)}.
    \end{equation}
Combining the above results and taking a square root yields the desired result
    \begin{equation}
        \left| \tilde{u} - \mathcal{I} \tilde{u} - w_h \right|_{H^1(\Omega)} \leq C h^{ p} \left\Vert \tilde{u}_0 - \tilde{u}_1 \right\Vert_{H^{p+1}(\Omega)}.
    \end{equation}
This completes the proof. 
\end{proof}

With the above Lemmas, we now present our main results on the errors. 

\begin{theorem}Suppose $u$ is the exact solution of problem~\eqref{eq:pde} and $u_h$ is HoSGFEM approximation of system~\eqref{eq:GalApproProm} with the enrichment space defined in \eqref{eq:impSGFEMp}, then we have
    \begin{equation}
        | u - u_h |_{H^1(\Omega)} \leq C h^p \Vert u \Vert_{H^{p + 1}(\Omega)}.
    \end{equation}
\end{theorem}
\begin{proof}Define $v_h :=  \mathcal{I}_h \tilde{u} + \mathcal{I}_h \tilde{u}_1 + w_h$, where $w_h \in V^h_{ENR}$ satisfies the estimation
given in Lemma \ref{lemma2}. It is clear that $v_h \in V^h$ since $\mathcal{I}_h \tilde{u} + \mathcal{I}_h \tilde{u}_1 \in V^h_{FEM}$ and
$w_h \in V^h_{ENR}$. Thus, we have
    \begin{equation}
        \begin{split}
            \left| u - v_h \right|_{H^1(\Omega)}
            & = \left|u -(\tilde{u} + \tilde{u}_1) + (\tilde{u} + \tilde{u}_1) - v_h\right|_{H^1(\Omega)}\\
            & \leq |u -(\tilde{u} + \tilde{u}_1)|_{H^1(\Omega)} + |(\tilde{u} + \tilde{u}_1) - v_h|_{H^1(\Omega)}\\
            & \leq |u -(\tilde{u} + \tilde{u}_1)|_{H^1(\Omega)} + |(\tilde{u} + \tilde{u}_1) - \mathcal{I}_h \tilde{u} - \mathcal{I}_h \tilde{u}_1 - w_h |_{H^1(\Omega)}\\
            & \leq |u -(\tilde{u} + \tilde{u}_1)|_{H^1(\Omega)} + |\tilde{u} - \mathcal{I}_h \tilde{u} - w_h|_{H^1(\Omega)}
            + |\tilde{u}_1 - \mathcal{I}_h \tilde{u}_1|_{H^1(\Omega)}\\
            & \leq C h^{p} \Vert \tilde{u}_0 - \tilde{u}_1 \Vert_{H^{p+1}(\Omega)} + C h^{p}\Vert \tilde{u}_1 \Vert_{H^{p+1}(\Omega)}
            \leq C h^{p}\Vert u \Vert_{H^{p+1}(\Omega)},
        \end{split}
    \end{equation}
    where the penultimate inequality follows from Lemma~\ref{lemma2} and the interpolation estimate, and the last inequality follows from
    the properties of $\tilde{u}_i$~\eqref{eq:propertyU}. Finally, according to the C\'ea's Lemma, we have the following estimation,
    \begin{equation}
       |u - u_h|_{H^1(\Omega)} \leq C \min\limits_{v \in V^h} |u - v|_{H^1(\Omega)} \leq  C |u - v_h|_{H^1(\Omega)} \leq C h^p \Vert u \Vert_{H^{p+1}(\Omega)},
    \end{equation}
    which gives the desired result and completes the proof.
\end{proof}

\section{Numerical experiments}\label{sec:num}
In this section, we present various numerical examples to demonstrate the effectiveness of the proposed method. 
We first consider the benchmark model problem~\eqref{eq:pde} with two interfaces studied in \cite{zhang2020stable, zhang2022condensed}. 
We show high-order optimal convergent results for high-order elements ($p=2,3,4,5$) and compare with the widely-used methods discussed in Section \ref{sec:m}. 
Subsequently, in Section~\ref{subsec:TestRobustness}, we focus on the study of robustness. To this end, we consider a challenging problem, an elastostatic problem in bimaterial solids.

For clarity in subsequent demonstration of the numerical study, we adopt the following abbreviations to denote several GFEMs discussed in Sections~\ref{sec:sgfem}.
\begin{itemize}
    \item FEM: Standard FEM of degree $p$;
    \item GFEM: Standard GFEM which employs $\phi^1$ as PU functions, $dx^iy^j, 0 \leq i+j < p$ as enrichment functions, and the enriched node set is
    all the vertices of interface elements;
    \item SGFEM: Stable GFEM which employs Hermite polynomials as PU functions, $\tilde{d}x^iy^j - \mathcal{I}_h(\tilde{d}x^iy^i)$, $0 \leq i+j < p$ as enrichment functions, and the enriched node set includes all the vertices of interface elements. The LPCA technology is used to reduce the stiffness matrix's SCN; see \cite{zhang2020stable};
    \item CGFEM: Corrected GFEM which employs $\phi^1$ as PU functions and $dR(x, y)x^iy^j$ as enrichment functions. Enriched node set is $I^1_{v}$, which is defined in~\eqref{eq:CGFEM2};
    \item HoSGFEM: High-order stable GFEM proposed in this paper which employs~\eqref{eq:PUfunction} as PU functions and $\tilde{d}(x-x_k)^i(y-y_k)^j - \mathcal{I}_h(\tilde{d}(x-x_k)^i(y-y_k)^j)$, $0 \leq i +j < p$ as enrichment functions. Enriched nodes correspond to the interface elements.
\end{itemize}

The main tests in the following are two classic interface geometries: a straight line and a circle. We present the numerical results
for model problem~\eqref{eq:pde} on these two interfaces. The computational domain is fixed as the unit square $[0,1] \times[0,1]$ and uniformly
discretized into $N\times N$ quadrilateral elements, where $N=5, 10, 20,\cdots$. Element size $h$ is $\frac{1}{N}$. Figure~\ref{fig:TwoInterfaces}
illustrates the case for $N=10$, with the interface elements highlighted in light blue.
\begin{figure}[htbp]
    \centering
    \hspace{3cm}
    \subfloat[Straight line interface]{\centering\includegraphics[width=0.25\linewidth]{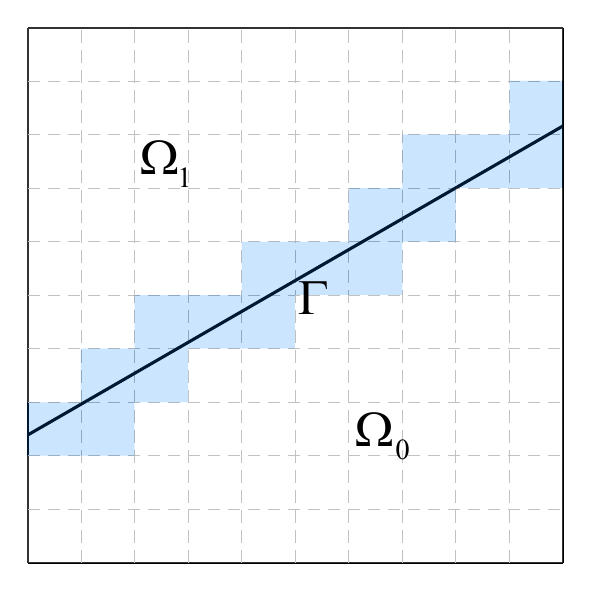}}
    \hfill
    \subfloat[Circular interface]{\centering\includegraphics[width=0.25\linewidth]{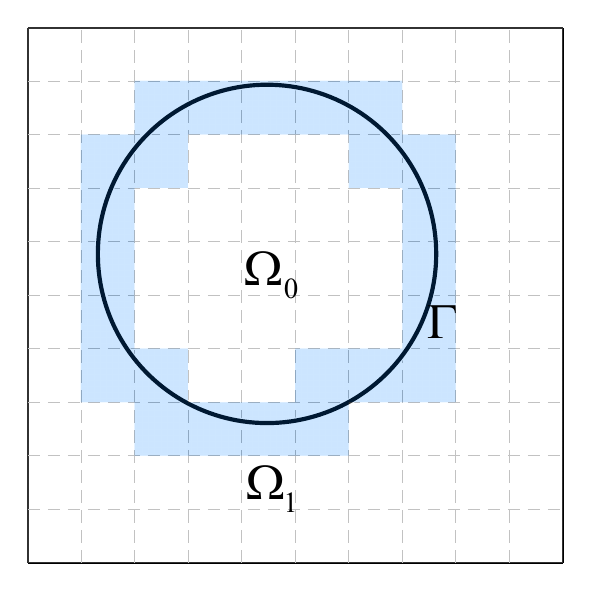}}
    \hspace{3cm}
    \caption{Two different interface geometries, with elements intersecting the interface highlighted in light blue\label{fig:TwoInterfaces}}
\end{figure}

Straight line interface $\Gamma$ has the equation $y = \tan{(\theta_0)}(x - 1 - \beta) + 1$, where $\theta_0 = \frac{\pi}{6}$ and $\beta = \frac{1}{\pi}$. The region below the line is denoted as $\Omega_0$ and the other is $\Omega_1$. 
We consider the problem \eqref{eq:pde} with a manufactured solution $u$ as
\begin{equation}\label{eq:ManuSol-Line}
    u= \begin{cases}r^{\frac{4}{3}} \cos \left(\frac{4}{3}\left(\theta+\pi-\theta_0\right)\right)+\frac{\kappa_0}{\kappa_1} r^{\frac{4}{3}} \sin \left(\frac{4}{3}\left(\theta+\pi-\theta_0\right)\right)+\sin (x y), & y>\tan \left(\theta_0\right)(x-1-\beta)+1, \\ r^{\frac{4}{3}} \cos \left(\frac{4}{3}\left(\theta+\pi-\theta_0\right)\right)+r^{\frac{4}{3}} \sin \left(\frac{4}{3}\left(\theta+\pi-\theta_0\right)\right)+\sin (x y), & y \leq \tan \left(\theta_0\right)(x-1-\beta)+1,\end{cases}
\end{equation}
where $(r, \theta)$ is the polar coordinate centered at $(1 + \beta, 1)$. $\kappa_0$ and $\kappa_1$ are the predefined coefficients to be specified later.

Circular interface has the center at $(x_0, y_0) = (\frac{1}{\sqrt{5}}, \frac{1}{\sqrt{3}})$ with radius $r_0 = \frac{1}{\sqrt{10}}$. The region inside the circle is denoted as $\Omega_0$ and the outside is $\Omega_1$. The manufactured solution is
\begin{equation}\label{eq:circle_u}
u= \begin{cases}\frac{2 \kappa_1}{ r_0^4} r^2 \cos (2 \theta), & r<r_0, \\
\frac{\kappa_1+\kappa_0}{ r_0^4} r^2 \cos (2 \theta)+ \frac{(\kappa_1-\kappa_0)}{r^2} \cos (2 \theta), & r \geq r_0,\end{cases}
\end{equation}
where $(r, \theta)$ is the polar coordinate centered at $(x_0, y_0)$. The remaining terms $f, q$, and $g$ in the model problem~\eqref{eq:pde} can be directly derived from the manufactured solution $u$.

Throughout all numerical experiments, we have the following observations
\begin{itemize}
  \item The numerical experiments verify the optimal error convergence analysis established in Section~\ref{sec:ea} for the proposed HoSGFEM of degrees 2 to 5.
  \item SGFEM can achieve optimal convergence rates only for degree $2$ by carefully selecting the LPCA parameters. Its convergence rate degrades for higher-order approximations.
  \item The accuracy of higher-order GFEM and CGFEM is fundamentally limited by linear dependence, which induces ill-conditioning that prevents convergence and can also cause the error to increase upon mesh refinement. 
  \item The FEM consistently achieves only a convergence rate of $\mO(h^{0.5})$.
\end{itemize}

\subsection{Error convergence rates and scaled condition numbers}
\label{subset:TestModedProb}
The first experiment focuses on testing the convergence property. We implement all the methods mentioned above in both interfaces with various values for $\kappa_0$ and $\kappa_1$. Figure \ref{fig:hosgfem_H1error} show the energy-norm ($\| w\|_E := \sqrt{a(w,w)}$) error convergence rates of the proposed HoSGFEM with $p=2,3,4,5$ for the interface problem \eqref{eq:pde} on three different interface configurations. 
The configurations are: (1) the straight line in Figure \ref{fig:TwoInterfaces} with $\kappa_0=1, k_1 = 10$; (2) the circular interface in Figure \ref{fig:TwoInterfaces} with $\kappa_0=1, k_1 = 20$; and (3) the circular interface in Figure \ref{fig:TwoInterfaces} with $\kappa_0=1, k_1 = 200$.
Figure \ref{fig:hosgfem_scn} shows the scaled condition numbers for these tests.
One observes that the HoSGFEM approximations converge optimally to the exact solution in the energy norm, confirming our theoretical findings established in Section \ref{sec:ea}.
The scaled condition number of the stiffness matrix exhibits order $\mO(h^{-2})$, validating the desired numerical stability on SCN. We observe a slight degeneracy in the convergence rates for higher-order elements with fine meshes. This is expected due to the increased condition number and round-off errors.
\begin{figure}[htbp]
    \centering
    \includegraphics[width=1\linewidth]{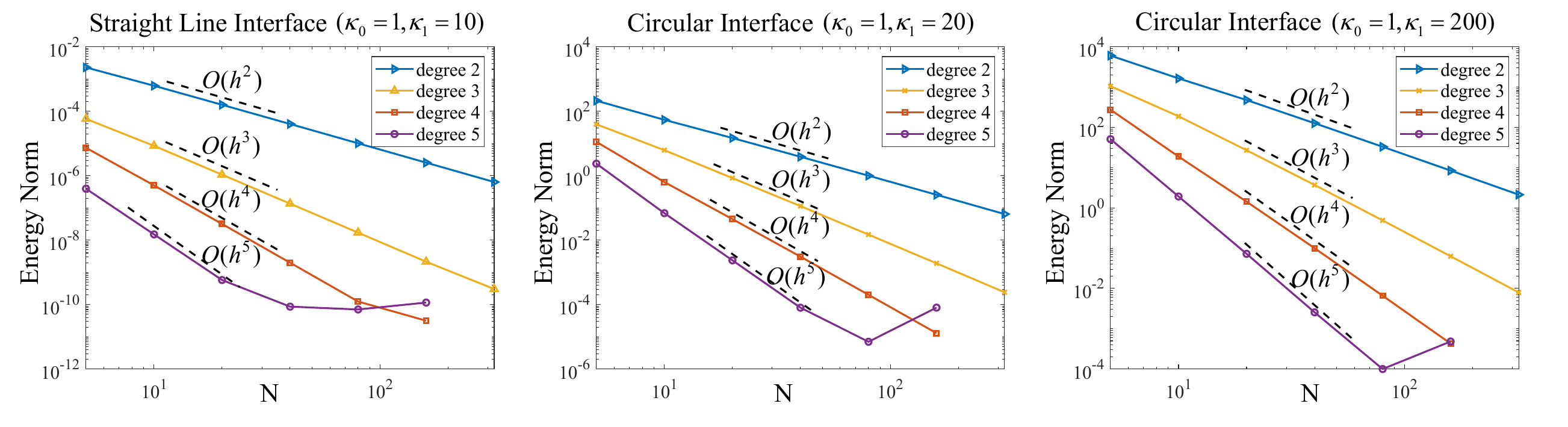}
    \caption{The energy-norm errors with respect to mesh size $N$ for HoSGFEM with $p=2,3,4,5$. The left plot shows the results for the straight line interface with $\kappa_0=1, k_1 = 10$; the middle plot shows the results for the circular interface with $\kappa_0 = 1, \kappa_1 = 20$; the right plot shows the results for the circular interface with $\kappa_0 = 1, \kappa_1 = 200$.}
    \label{fig:hosgfem_H1error}
\end{figure}

\begin{figure}[htbp]
    \centering
    \includegraphics[width=1\linewidth]{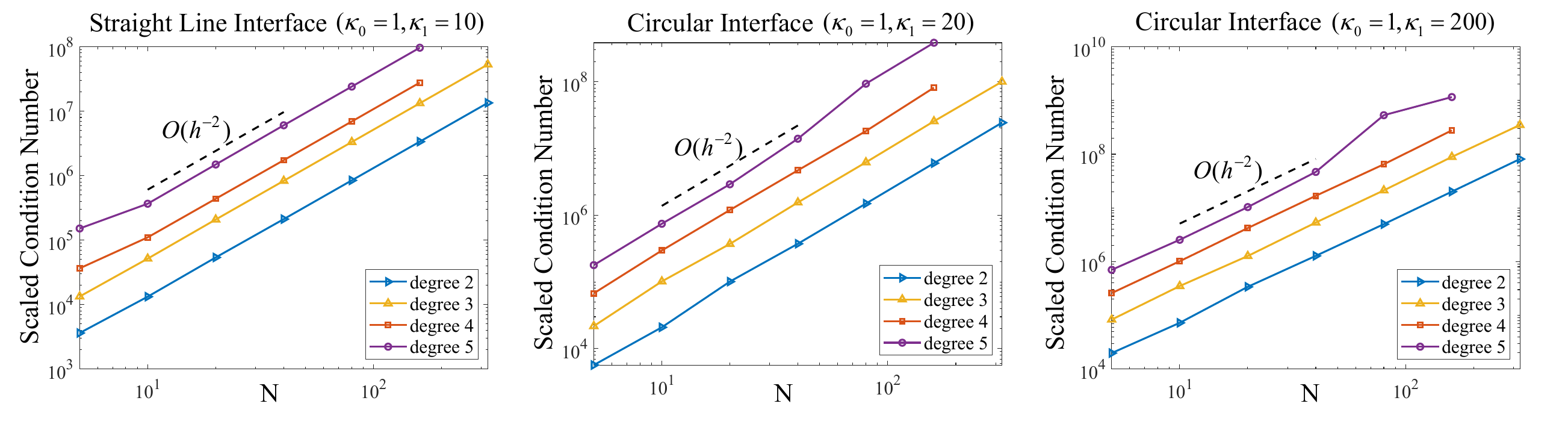}
    \caption{The scaled condition number with respect to $N$ for HoSGFEM with $p=2,3,4,5$. The left plot shows the results for the straight line interface with $\kappa_0=1, k_1 = 10$; the middle plot shows the results for the circular interface with $\kappa_0 = 1, \kappa_1 = 20$; the right plot shows the results for the circular interface with $\kappa_0 = 1, \kappa_1 = 200$.}
    \label{fig:hosgfem_scn}
\end{figure}

 Figures \ref{fig:line_e1error}, \ref{fig:line_scn}, \ref{fig:circle_e1error}, and \ref{fig:circle_scn} present a detailed comparison among different methods and polynomial degrees. 
 FEM consistently achieves only $\mO(h^{0.5})$ convergence rate, while its corresponding scaled condition number grows at a rate of $\mO(h^{-2})$. GFEM and CGFEM improve approximation accuracy by incorporating combinations of higher-order polynomials and distance functions. However, near-optimal convergence can only be observed in the first few levels of mesh refinement. 
 The scaled condition numbers of the stiffness matrices obtained by both methods are exceptionally large. This can be regarded as a key reason for the rapid stagnation of error reduction. 
 SGFEM employs the LPCA technology to process the stiffness matrix. In all experiments, we choose $\epsilon = 10^{-15}$ as the threshold. Owing to the additional stabilization techniques employed, the scaled condition number of SGFEM is less than that of GFEM and CGFEM. However, the removal of certain basis function components leads to a convergence rate lower than the theoretical value. Moreover, our experiments indicate that the choice of the threshold parameter $\epsilon$ significantly influences the final convergence results. Therefore, this parameter must be carefully selected to simultaneously ensure overall stability and approximation accuracy. However, such a parameter-dependent strategy for guaranteeing performance may pose challenges in practical applications. 
 Among these methods, HoSGFEM demonstrates superior performance, maintaining the same growth rate of the scaled condition number as FEM while achieving optimal convergence results that are in agreement with the theoretical findings.
 \begin{figure}[htbp]
    \centering
    \includegraphics[width=1\linewidth]{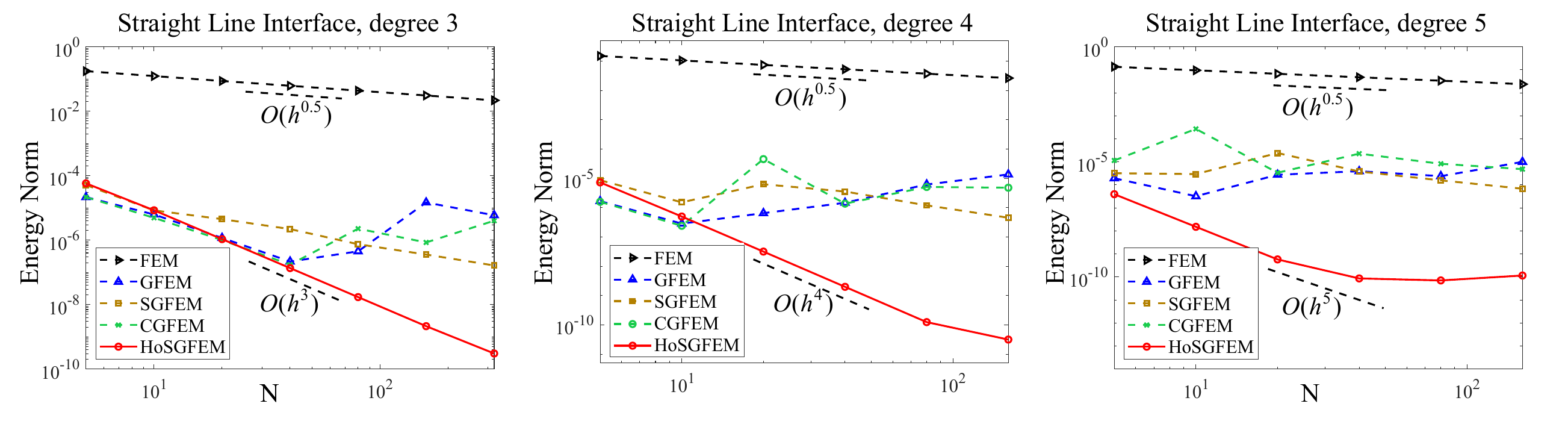}
    \caption{Comparison of energy-norm error for different GFEMs on a straight line interface, with degrees 3, 4, and 5 (from left to right).}
    \label{fig:line_e1error}
\end{figure}

\begin{figure}[htbp]
    \centering
    \includegraphics[width=1\linewidth]{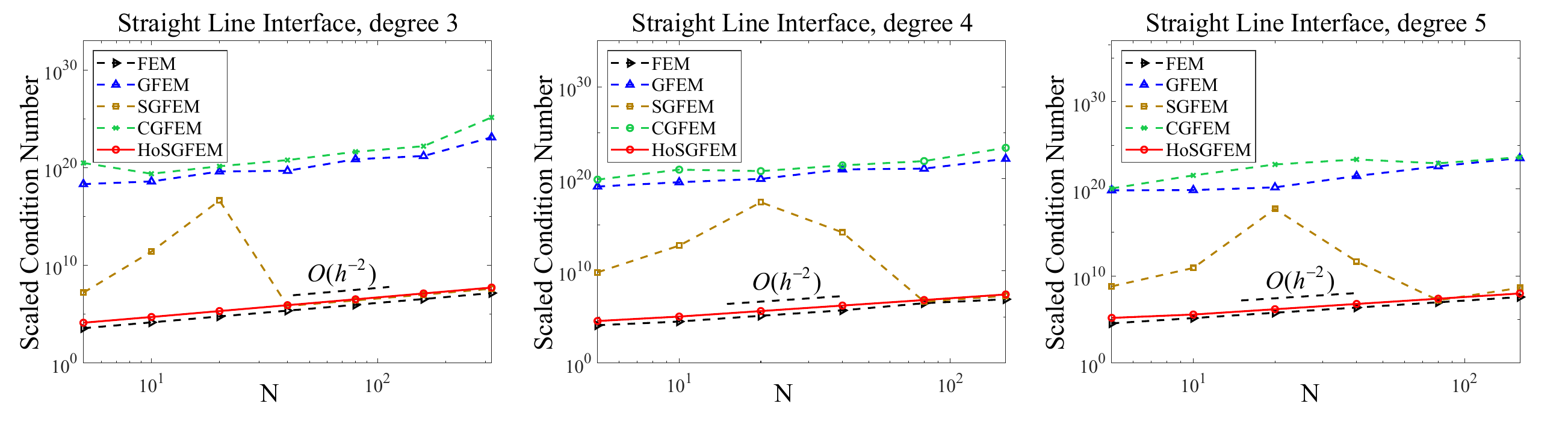}
    \caption{Comparison of scaled condition number for different GFEMs on a straight line interface, with degrees 3, 4, and 5 (from left to right).}
    \label{fig:line_scn}
\end{figure}

\begin{figure}[htbp]
    \centering
    \includegraphics[width=1\linewidth]{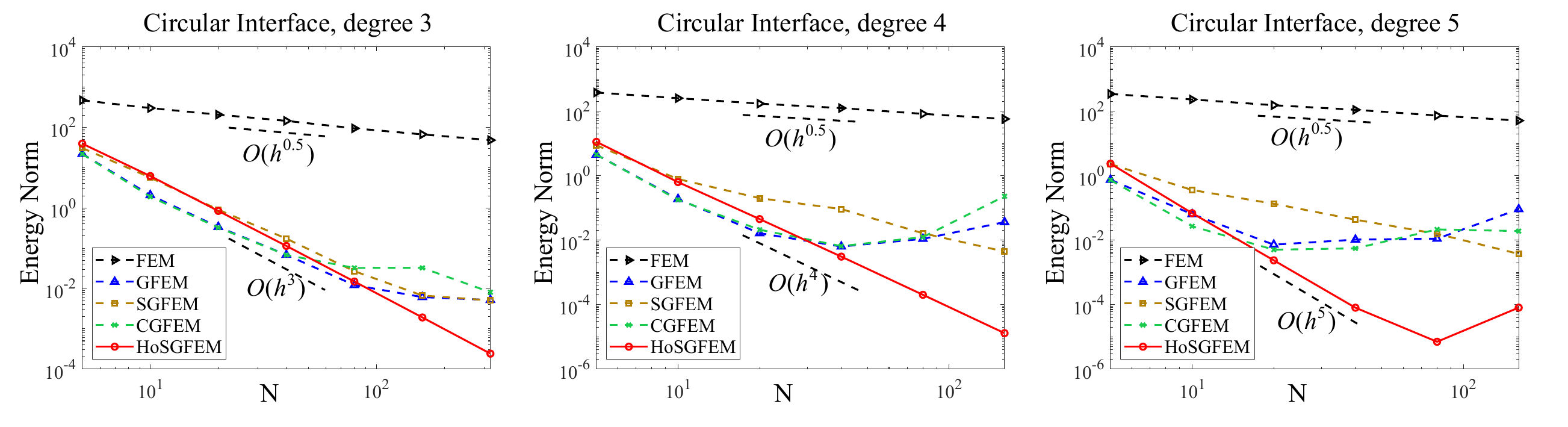}
    \caption{Comparison of energy-norm error for different GFEMs on a circular interface, with degrees 3, 4, and 5 (from left to right).}
    \label{fig:circle_e1error}
\end{figure}

\begin{figure}[htbp]
    \centering
    \includegraphics[width=1\linewidth]{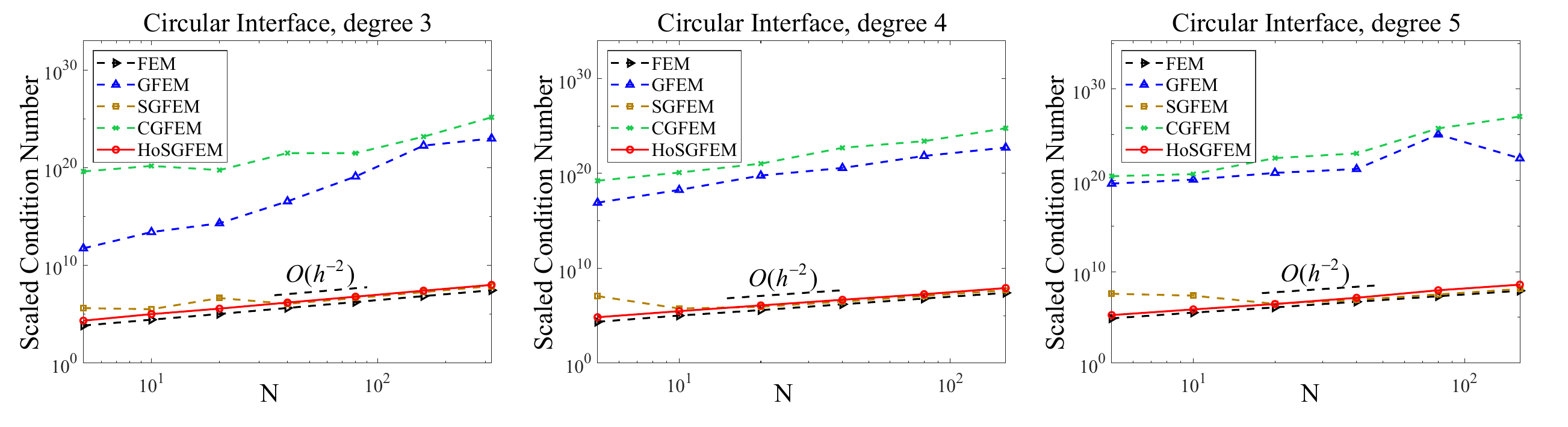}
    \caption{Comparison of scaled condition number for different GFEMs on a circular interface, with degrees 3, 4, and 5 (from left to right).}
    \label{fig:circle_scn}
\end{figure}

\subsection{Robustness test}\label{subsec:TestRobustness}
Next, we investigate how the SCN of the stiffness matrix changes when the interface approaches the element boundary. 
For this purpose, we use a fixed 40$\times$40 uniform quadrilateral mesh in the experiment, with the interface positioned as a straight line parallel to the $x$ axis at $y = 0.5 + \delta$. 
As $\delta$ decreases, the interface progressively approaches the element boundary. Taking $\delta = 0.03 \times 2^{-i}(i = 1, 2, 3,\ldots, 15)$, Figure \ref{fig:RobustDeltaTest} shows the stiffness matrix's SCN of different GFEMs for cases with $p=3,4,5$. The figure demonstrates that as the interface approaches the element boundary, the stiffness matrix's SCN of GFEM and CGFEM consistently maintain significantly large values, while FEM, SGFEM, and HoSGFEM remain at relatively low values. It is worth mentioning that when $\delta$ becomes sufficiently small, the scaled condition number of HoSGFEM exhibits some growth. This phenomenon may be due to the potential generation of extremely small enrichment functions in $d x^iy^j -\mathcal{I}_h(dx^iy^j)$, which could introduce considerable computational errors in the basis functions.

\begin{figure}[htbp]
    \centering
    \includegraphics[width=1\linewidth]{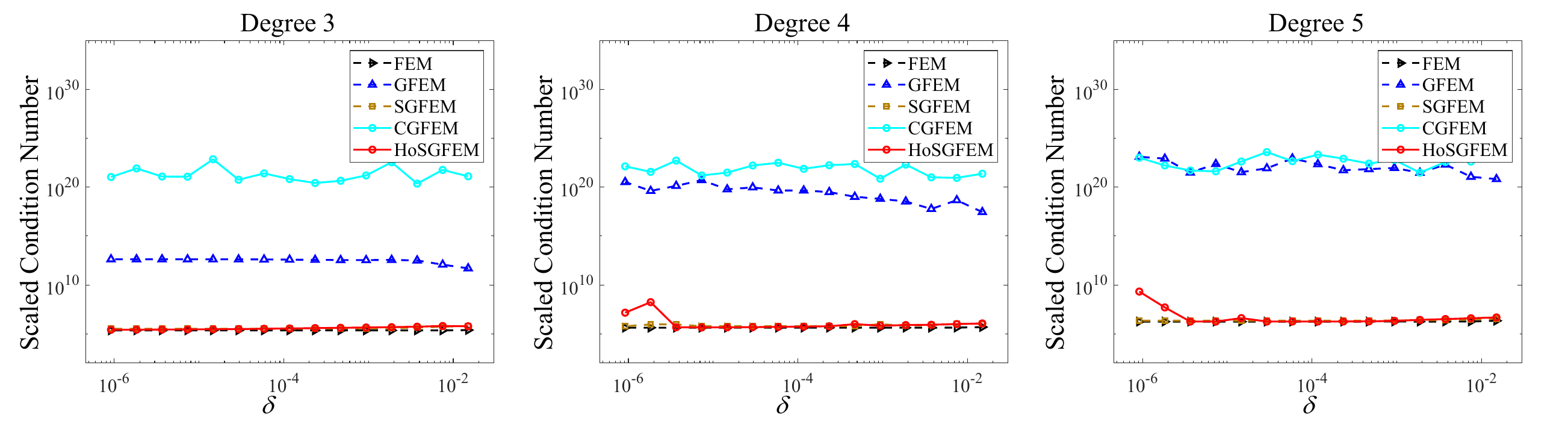}
    \caption{The SCN of stiffness matrices for different GFEMs with varying $\delta$. With $N = 40$ fixed, the left plot shows the case of cubic elements, the middle plot shows the quartic case, and the right plot shows the quintic case.}\label{fig:RobustDeltaTest}
\end{figure}

\subsection{Application: bi-material with a circular inclusion}\label{subsec:TestElastostaitcProb}
We now consider an elastostatic problem studied in \cite{sukumar2001modeling}. Figure \ref{fig:BiMaterialProb} illustrates a disk of radius $b$ containing a circular inclusion of radius $a$. The outer annular region $\Omega_2$ has material properties with Lam{\'e} constants $\lambda_2=5.7692$ and $\mu_2 = 3.8461$ and the inner circle has parameters $\lambda_1 =0.4$ and $\mu_1=0.4$. Impose the linear displacement field $\boldsymbol{u} = (x_1, x_2)$ on $\partial \Omega(r =b)$, which means points $\boldsymbol{x} = (x_1, x_2) \in \partial \Omega$ are constrained to move to the deformed positions ${\bar{\boldsymbol{x}} = (2 x_1, 2 x_2)}$.
\begin{figure}[htpb]
    \centering
    \includegraphics[width=0.35\linewidth]{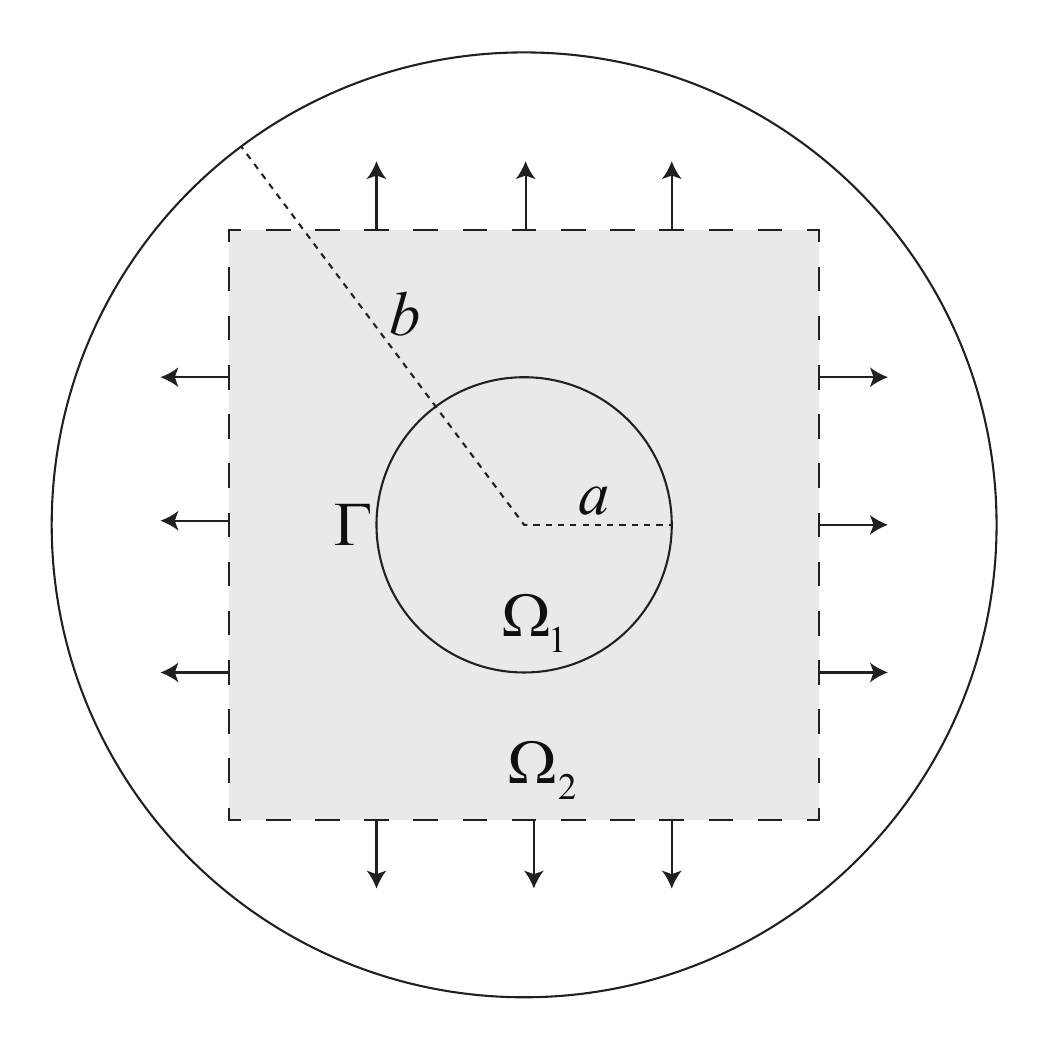}
    \caption{Bi-material with a circular inclusion, the gray area is the computation domain.}
    \label{fig:BiMaterialProb}
\end{figure}

The corresponding field equations of elastostatics are
\begin{equation*}
\begin{aligned}
    & \nabla \cdot \boldsymbol{\sigma}+\boldsymbol{f}=0,\ \text { in } \Omega, \\
    & \boldsymbol{\sigma}=\mathbf{C}: \boldsymbol{\varepsilon}, \\
    & \boldsymbol{\varepsilon}=\frac{1}{2} (\nabla \boldsymbol{u} + (\nabla \boldsymbol{u})^T),\\
\end{aligned}
\end{equation*}
where $\boldsymbol{u}$, $\boldsymbol{f}$, and $\boldsymbol{C}$ denote the displacement field, body force, and elasticity tensor, respectively. $\boldsymbol{\epsilon}$ and $\boldsymbol{\sigma}$ represent the strain tensor and stress tensor. Additionally, the displacement boundary conditions and traction boundary conditions are specified as follows
\begin{equation*}
\begin{aligned}
& \boldsymbol{u}=\boldsymbol{u}_0 \quad \text { on } \Gamma_u, \\
& \boldsymbol{\sigma} \cdot \boldsymbol{n}=\boldsymbol{t}_0 \quad \text { on } \Gamma_t, \\
& \llbracket \boldsymbol{\sigma} \cdot \boldsymbol{n}_I^i \rrbracket=0 \quad \text { on } \Gamma_I,
\end{aligned}
\end{equation*}
where $\boldsymbol{u}_0$ and $\boldsymbol{t}_0$ denote the prescribed displacement and traction, respectively, and $\Gamma_I$ represents the material interface. Based on the governing equations above, the exact solution for the bimaterial problem is given by
\begin{equation*}
    \begin{aligned}
& u_r(r)= \begin{cases}{\left[\left(1-\frac{b^2}{a^2}\right) \alpha+\frac{b^2}{a^2}\right] r,} & 0 \leqslant r \leqslant a, \\
\left(r-\frac{b^2}{r}\right) \alpha+\frac{b^2}{r}, & a \leqslant r \leqslant b,\end{cases} \\
& u_\theta=0
\end{aligned}
\end{equation*}
where
\begin{equation*}
\alpha=\frac{\left(\lambda_1+\mu_1+\mu_2\right) b^2}{\left(\lambda_2+\mu_2\right) a^2+\left(\lambda_1+\mu_1\right)\left(b^2-a^2\right)+\mu_2 b^2}.
\end{equation*}
Further details regarding this problem can be found in~\cite{sukumar2001modeling}. The computational domain is defined as a square region with side length 2 (gray area in Figure \ref{fig:BiMaterialProb}), where exact traction conditions are applied on its boundaries. Exact displacements in the $x$ and $y$ directions is shown in Figure \ref{fig:BiMaterialExactU}. We discretize the domain using a uniform $N \times N$ mesh and choose $a = 0.42$ and $b = 2$ for the numerical experiments. Energy norm is computed by
\begin{equation*}
    \Vert \boldsymbol{u} - \boldsymbol{u}_h \Vert^2_{\mathcal{E}(\Omega)} := \int_{\Omega} \left(\boldsymbol{\varepsilon} - \boldsymbol{\varepsilon}_h\right)^{T} : C :\left(\boldsymbol{\varepsilon} - \boldsymbol{\varepsilon}_h\right)d\Omega.
\end{equation*}

Figure \ref{fig:BiMaterialEnergyError} and \ref{fig:BiMaterialSCN} presents the error results of different methods for degrees from 2 to 4. Similar to results in Section \ref{subset:TestModedProb}, FEM remains only $\mO(h^{0.5})$ convergence rate. Both GFEM and CGFEM exhibit large scaled condition numbers in these cases. In particular, for the cubic and quartic cases, the energy error in CGFEM initially decreases but subsequently exhibits a rapid increase. SGFEM maintains optimal convergence only in the case of degree 2. At higher degrees, the corresponding convergence rate fails to achieve optimality. 
Overall, HoSGFEM demonstrates superior performance, with numerical results closely matching the theoretically optimal convergence rate while maintaining the desired $\mO(h^{-2})$ growth rate of the scaled condition number.

\begin{figure}[htpb]
    \centering
    \includegraphics[width=0.75\linewidth]{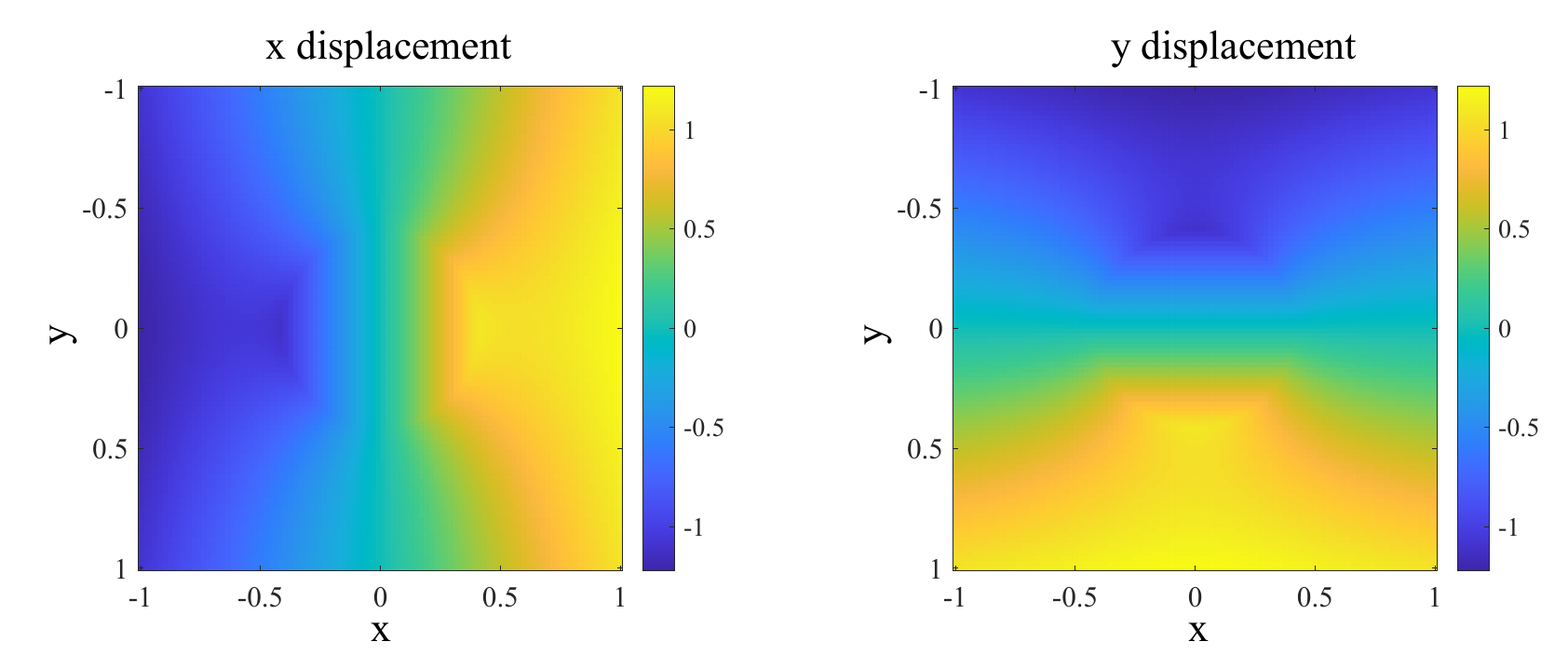}
    \caption{Exact displacements in both $x$- and $y$-directions.}
    \label{fig:BiMaterialExactU}
\end{figure}
\begin{figure}[htpb]
    \centering
    \includegraphics[width=1\linewidth]{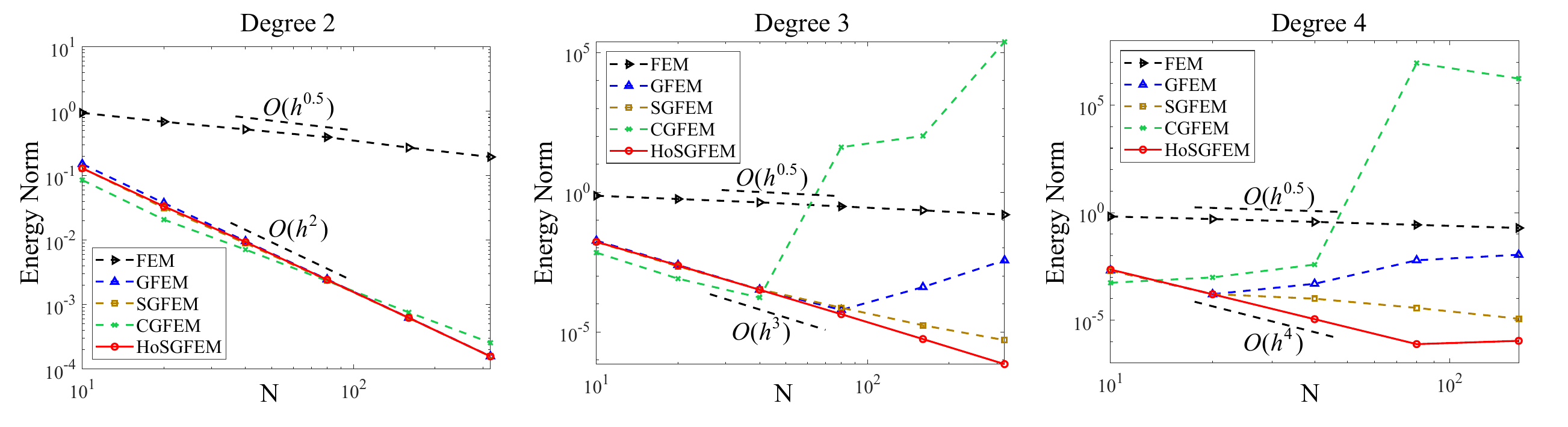}
    \caption{The energy-norm errors with respect to $N$ of different GFEMs for the bi-material problem.}
    \label{fig:BiMaterialEnergyError}
\end{figure}

\begin{figure}[htpb]
    \centering
    \includegraphics[width=1\linewidth]{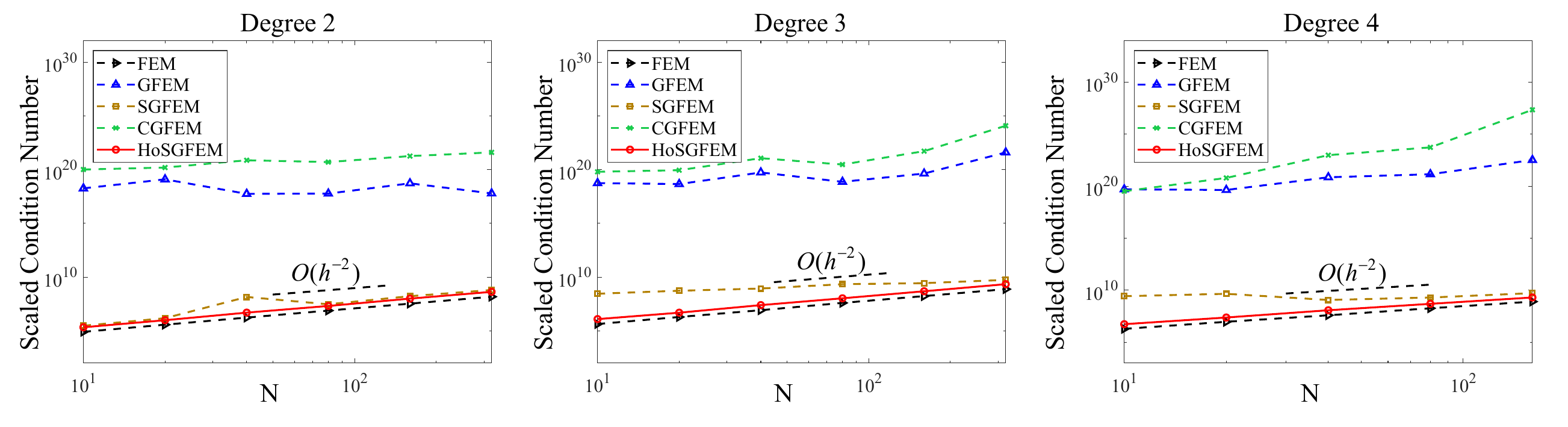}
    \caption{The scaled condition number with respect to $N$ of different GFEMs for the bi-material problem.}
    \label{fig:BiMaterialSCN}
\end{figure}

\section{Conclusions}\label{section7}
In this paper, we developed a high-order stable GFEM for interface problems. 
The proposed method satisfies the desired properties of (a) optimal convergence, (b) numerical stability, and (c) robustness. 
Similar to existing stable GFEM~\cite{babuvska2012stable, zhang2020stable} for lower-order elements, we utilized the distance function and interpolation operator in the construction of the enrichment function. In contrast to stable GFEM of degree two, where one uses Local Principal Component Analysis for removing correlated components from the stiffness matrix, our approach ensures stable growth of the scaled condition number, which is inherited from the construction of the enrichment space. Specifically, we designed partition of unity functions based on $p$-th order Bernstein polynomials, where the enrichment space uses interface elements rather than vertices as enriched nodes, thereby reducing the degrees of freedom. Furthermore, we employed $(x-x_k)^i(y-y_k)^j$ instead of $x^iy^j$ when constructing enrichment functions and performing local orthogonalization at each enriched node. The entire process is explicit and does not depend on any auxiliary parameters.
The method's performance is validated through numerical experiments for $p\le 5$. Results confirm optimal convergence rates, consistent with theory, and a stiffness matrix whose scaled condition number remains well-behaved at $\mathcal{O}(h^{-2})$ as for standard FEM.
For elements of order $p\ge 6$, the proposed method may encounter larger errors from rounding errors and the local orthogonalization Algorithm \ref{alg:gs} that leads to lower robustness. 
Special treatment might be necessary as the interface approaches the element boundary. 

While the present study demonstrates the efficacy of the proposed high-order stable GFEM, it opens up numerous avenues for further research. The immediate next step is a rigorous extension to three-dimensional simulations, which will serve as the ultimate test of its geometric flexibility and stabilization scheme. Beyond this, we envision applying the methodology to more sophisticated interface problems, particularly in coupled systems like thermo-elasticity or electro-diffusion, where high-order accuracy on unfitted meshes could offer significant advantages. From a computational standpoint, a critical future direction is the creation of specialized iterative solvers and multi-level preconditioners to address the computational cost associated with high-order discretizations and cut cells, thereby enabling large-scale scientific and industrial applications.

\section*{Acknowledgements}
\vspace{-0.4cm} The authors were supported by the Strategic Priority Research Program of the Chinese Academy of
Sciences Grant No. XDB0640000, the Key Grant Project of the NSF of China (No.12494552), the NSF of China (No.12471360) and start-up funding from the Yau Mathematical Sciences Center, Tsinghua University.

\bibliographystyle{elsarticle-num-names}
\bibliography{refs}

\end{document}